\documentclass[11pt,a4paper]{amsart}
\usepackage[utf8]{inputenc}
\usepackage{amsmath}
\usepackage{amsfonts}
\usepackage{amscd}
\usepackage{enumerate}
\usepackage[margin= 0.95in]{geometry}
\usepackage{amsthm}
\usepackage{mathrsfs}
\usepackage{amssymb}
\usepackage[all]{xy}
\usepackage{xcolor}
\usepackage{graphics}
\usepackage{lscape}
\usepackage{array}
\usepackage{ulem}
\usepackage{setspace}
\spacing{1.20}
\usepackage{xy}

\usepackage{amssymb} %More math symbols
\usepackage{amsmath} %Enhances display of math stuff
\usepackage{amsthm} %Theorems, lemmas, etc
\usepackage{mathrsfs} %Calligraphic letters
\usepackage{mathtools} %Enhances display of math stuff
\usepackage{tikz-cd} %Commutative diagrams
\usetikzlibrary{graphs,decorations.pathmorphing,decorations.markings}
\usepackage{stmaryrd} %mapsfrom arrow
\usepackage{upgreek} %To use upkappa
\usepackage{centernot} %To negate \impliedby
\include{commands}
\usepackage{comment}
\usepackage[shortlabels]{enumitem}
\numberwithin{equation}{section}
\newtheorem*{theorem*}{Theorem}
\newtheorem*{theorem_A}{Theorem A}
\newtheorem*{theorem_B}{Theorem B}
\newtheorem{theorem}{Theorem}[section]
\newtheorem{lemma}[theorem]{Lemma}
\newtheorem{proposition}[theorem]{Proposition}
\newtheorem{corollary}[theorem]{Corollary}
\theoremstyle{definition}

\theoremstyle{remark}
\newtheorem{remark}[theorem]{Remark}

\usepackage{hyperref}\hypersetup{colorlinks}

\usepackage{color} %\textcolor{red}{Test}

\definecolor{darkred}{rgb}{1,0,0} %can change the intensity in [0,1]
\definecolor{darkgreen}{rgb}{0,1,0}
\definecolor{darkblue}{rgb}{0,0,1}

\hypersetup{colorlinks,
linkcolor=darkblue,
filecolor=darkgreen,
urlcolor=darkred,
citecolor=darkgreen}

\DeclareMathOperator{\Hom}{Hom}

\DeclareMathOperator{\Nm}{Nm}
\DeclareMathOperator{\Jac}{J}
\DeclareMathOperator{\Prym}{P}
\DeclareMathOperator{\id}{1}

\DeclareMathOperator{\Pic}{Pic}
\DeclareMathOperator{\tot}{tot}
\DeclareMathOperator{\Coh}{Coh}
\DeclareMathOperator{\Irr}{Irr}

\DeclareMathOperator{\Hilb}{Hilb}

\DeclareMathOperator{\Higgs}{\mathfrak{Higgs}}
\DeclareMathOperator{\GL}{GL}
\DeclareMathOperator{\SL}{SL}
\DeclareMathOperator{\PGL}{PGL}

\newcommand{\Aa}{\mathcal{A}}
\newcommand{\Bb}{\mathcal{B}}

\newcommand{\Dd}{\mathcal{D}}
\newcommand{\Ee}{\mathcal{E}}
\newcommand{\Ff}{\mathcal{F}}
\newcommand{\Gg}{\mathcal{G}}
\newcommand{\Ii}{\mathcal{I}}
\newcommand{\Hh}{\mathcal{H}}

\newcommand{\Kk}{\mathcal{K}}

\newcommand{\Mm}{\mathcal{M}}
\newcommand{\Nn}{\mathcal{N}}
\newcommand{\Oo}{\mathcal{O}}
\newcommand{\Pp}{\mathcal{P}}
\newcommand{\Qq}{\mathcal{Q}}
\newcommand{\Rr}{\mathcal{R}}

\newcommand{\Tt}{\mathcal{T}}
\newcommand{\Uu}{\mathcal{U}}
\newcommand{\Vv}{\mathcal{V}}
\newcommand{\Ww}{\mathcal{W}}

\newcommand{\Hhom}{\mathcal{H}\!om}

\newcommand{\p}{\mathrm{p}}
\newcommand{\q}{\mathrm{q}}

\newcommand{\BBB}{\mathrm{BBB}}
\newcommand{\BAA}{\mathrm{BAA}}
\newcommand{\ol}[1]{\overline{#1}}

\newcommand{\wt}[1]{\widetilde{#1}}

\newcommand{\ZZ}{\mathbb{Z}}

\newcommand{\QQ}{\mathbb{Q}}

\newcommand{\into}{\hookrightarrow}

\newcommand{\adjointpair}[2]{{#1}\dashv {#2}}

\newcommand{\quotient}[2]{{\raisebox{.2em}{\thinspace $#1$}\left / \raisebox{-.15em}{ $#2$}\right.}}

\newcommand\Quotient[2]{
\mathchoice
{% \displaystyle
\text{\raise1ex\hbox{\thinspace $#1$}\Big{/} \lower1ex\hbox{$#2$} \thinspace}%
}
{% \textstyle
#1\,/\,#2
}
{% \scriptstyle
#1\,/\,#2
}
{% \scriptscriptstyle  
#1\,/\,#2
}
}

\newcommand\GIT[2]{
\mathchoice
{% \displaystyle
\text{\raise1ex\hbox{\thinspace $#1$}\Big{/}\!\!\!\!\Big{/} \lower1ex\hbox{$#2$} \thinspace}%
}
{% \textstyle
#1\,/\,#2
}
{% \scriptstyle
#1\,/\,#2
}
{% \scriptscriptstyle  
#1\,/\,#2
a       }
}

\newcommand{\map}[5]{\begin{array}{ccc}   #1  & \stackrel{#5}{\longrightarrow} &  #2  \\  #3 & \longmapsto & #4  \end{array}}

\newcommand{\morph}[6]{\begin{array}{cccc} #6: & #1  & \stackrel{#5}{\longrightarrow} &  #2  \\ & #3 & \longmapsto & #4  \end{array}}

\title[Fourier--Mukai for compactified Pryms]{\bf Fourier--Mukai transform for fine compactified Prym varieties}

\author[E. Franco]{Emilio Franco}
\address{E. Franco,
\newline\indent Universidad Aut\'onoma de Madrid 
\newline\indent and
\newline\indent Instituto de Ciencias Matem\'aticas (CSIC--UAM--UCM--UC3M)
\newline\indent Campus de Cantoblanco 28049, Madrid, Espa\~na.}
\email{emilio.franco@uam.es}

\author[R. Hanson]{Robert Hanson}
\address{Robert Hanson, \newline\indent Centro de An\'alise Matem\'atica, Geometria e Sistemas Din\^{a}micos, 
\newline\indent Instituto Superior T\'ecnico, Universidade de Lisboa, 
\newline\indent Av. Rovisco Pais s/n, 1049-001 Lisboa, Portugal}
\email{robert.hanson@tecnico.ulisboa.pt}

\author[J. Ruano]{Jo\~ao Ruano}
\address{Jo\~ao Ruano, \newline\indent Centro de An\'alise Matem\'atica, Geometria e Sistemas Din\^{a}micos, 
\newline\indent Instituto Superior T\'ecnico, Universidade de Lisboa, 
\newline\indent Av. Rovisco Pais s/n, 1049-001 Lisboa, Portugal}
\email{joao.ruano@tecnico.ulisboa.pt}

\date{\today}

\thanks{
First author partially supported by the Spanish Ministry of Science and Innovation, through the ‘Severo Ochoa Programme for Centres of Excellence in R$\&$D’ (CEX2019-000904-S), and through project PGC2022-1001150218, and by the Portuguese FCT through CEECIND/04153/2017. Second author supported by La Caixa INPhINIT programme with fellowship number 1801P.01033.
}
\begin{document}

\begin{abstract}
Consider a finite flat morphism $\beta : C \to X$ between a reduced, projective, locally planar curve $C$ and a smooth projective curve $X$. Associated to two generic polarizations $q$ and $q'$ on $C$, one can construct the corresponding compactified Prym varieties $\ol{\Prym}_\beta(q)$ and $\ol{\Prym}_\beta(q')$. With $n$ the rank of $\beta$, the finite group $\Gamma = \Jac_X[n]$ of $n-$torsion line bundles acts on $\ol{\Prym}_\beta(q')$ by tensorisation. In this article we construct a Fourier--Mukai transform between the derived category of $\ol{\Prym}_\beta(q)$ and the $\Gamma$-equivariant derived category of $\ol{\Prym}_\beta(q')$, providing a derived equivalence between the $\SL_n$-Hitchin fibre and the dual $\PGL_n$-Hitchin fibre for a dense class of singular spectral curves. This is an extension of Fourier--Mukai transforms on compactified Jacobian varieties constructed by Arinkin and Melo--Rapagnetta--Viviani, which correspond to autoduality of $\GL_n$-Hitchin fibres.
\end{abstract}

\maketitle

\tableofcontents

\section{Introduction}

\subsection{Background}
The seminal Fourier--Mukai transform of Mukai \cite{mukai} defines a derived equivalence between an abelian variety $A$ and its dual $A^{\vee} = \Pic^0(A)$. Using the Poincar\'e bundle $\Pp$ on $A \times A^{\vee}$, the universal family for $\Pic^0(A)$, and the canonical projections $A \xleftarrow{p_1} A \times A^{\vee} \xrightarrow{p_2} A^{\vee}$, the derived equivalence is 
\begin{equation} \label{eq mukai}
\Phi^{\Pp}_{A \to A^\vee}: D^b(A) \longrightarrow D^b(A^{\vee}), \quad \Ff^\bullet \mapsto Rp_{2,*}(p_1^{*}\Ff^\bullet \otimes^L \Pp).
\end{equation}
On a smooth irreducible projective curve $C$ over an algebraically closed field, the Jacobian variety $\Jac_C$ is classically known to be a self-dual abelian variety: $\Jac_C \cong \Jac_C^{\vee}$. Mukai's work therefore provides an autoequivalence on the derived category of $\Jac_C$, 
\begin{equation}
\label{jac autodual}
\Phi^{\Pp}_{\Jac_C \to \Jac_C} : D^b(\Jac_C) \longrightarrow D^b(\Jac_C),
\end{equation}
which interchanges skyscraper sheaves and line bundles on $\Jac_C$. This can be viewed as a derived interpretation of the self-duality of $\Jac_C$. 

When $C$ has singularities $\Jac_C$ is no longer compact. Picking a generic polarization $q$ on $C$, a natural compactification is given by the moduli space $\ol \Jac_C(q)$ of $q$-stable pure dimension $1$ sheaves on $C$. The construction and study of such compactifications has been addressed by many authors, from which we highlight the work of Altman, Arinkin, Casalania-Martin, Esteves, Gagn\'e, Kass, Kleiman, Melo, Oda, Rapagnetta, Seshadri, Simpson and Viviani \cite{altman&kleiman, arinkin1, esteves, esteves&gagne&kleiman, esteves&kleiman, oda&seshadri, simpson,  melo0}. 

It is then natural to ask if the Fourier--Mukai transform \eqref{jac autodual} extends to $\ol \Jac_C(q)$. When $C$ is an integral planar curve, this was answered by Arinkin \cite{arinkin} by constructing a Poincar\'e sheaf - an extension of the Poincar\'e bundle, which defines a Fourier--Mukai transform. This was extended by Melo--Rapagnetta--Viviani \cite{melo1, melo2} to the case where $C$ is planar and reduced, obtaining a new Poincar\'e sheaf $\Pp$ on $\ol \Jac_C(q) \times \ol \Jac_C(q')$ and a Fourier-Mukai transform 
\begin{equation} \label{eq FM on compactified Jac}
\Phi^{\Pp}_{\ol\Jac \to \ol\Jac}: D^b(\ol \Jac_C(q)) \xrightarrow{\cong} D^b(\ol \Jac_C(q')).
\end{equation}
A similar story can be told after replacing Jacobian varieties with their Prym subvarieties. Given a finite $n:1$ morphism $\beta:C \to X$ between smooth curves, the corresponding Prym variety is an abelian subvariety $\Prym_\beta \subset \Jac_C$ defined as the (connected component of the identity of the) kernel of the Norm map $\Nm_\beta := \det \circ \beta_*$. It is well known that the dual abelian variety $\Prym_\beta^{\vee} = \Pic^0(\Prym_\beta^{\vee})$ is isomorphic to a quotient of $\Prym_{\beta}$ by a finite group $\Gamma$ consisting of $n$-torsion line bundles on $X$ (see for instance \cite{hausel&thaddeus}). Once more, there exists a Poincar\'e bundle $\Qq$ over $\Prym_\beta \times \Prym_\beta/\Gamma$ that provides an equivalence of categories
\begin{equation} \label{eq FM for Prym}
\Phi^{\Qq}_{\Prym_\beta \to \Prym_\beta/\Gamma} : D^b(\Prym_\beta) \longrightarrow D^b(\Prym_\beta/\Gamma).
\end{equation}
If we now allow $C$ to be singular, planar and reduced, and $X$ remains smooth, then Hausel and Pauly \cite{hausel&pauly} showed that one can define the Norm map by $\Nm_\beta = \det \circ R\beta_*$. As noted by Carbone \cite{carbone}, the Norm map extends to a morphism defined on the whole compactified Jacobian $\Nm_\beta : \ol\Jac_C(q) \to \Jac_X$ and one can naturally define the compactified Prym variety $\ol\Prym_\beta(q) := \Nm^{-1}(\Oo_X) \subset \ol{\Jac}_C(q)$. Using the aforementioned ideas of Arinkin and Melo--Rapagnetta--Viviani, this paper addresses the extension and modification of \eqref{eq FM for Prym} to the compactification.  

\subsection{Motivation}
\label{motivation}
Our motivation arises from a conjecture on the moduli stack $\Higgs_G$ that parameterises $G$-Higgs bundles over a smooth projective curve $X$. Given a pair of Langlands dual groups $G$ and $^L G$, Donagi--Pantev \cite{donagi&pantev} conjecture that the semi-classical limit of the Geometric Langlands correspondence provides an equivalence of categories 
\begin{equation} \label{geometric langlands}
D^b(\Higgs_G) \longrightarrow D^b(\Higgs_{^LG}).
\end{equation}
%After the passage to the moduli scheme $\Mm_G$ of semistable Higgs bundles, the semi-classical limit
In the context of Strominger--Yau--Zaslow mirror symmetry, this is expected to restrict to a Fourier--Mukai transform on the fibres of dual Hitchin fibrations $\mathfrak{h}_G : \Higgs_{G} \to B_G$ and $\mathfrak{h}_{^LG} : \Higgs_{^LG} \to B_{^LG}$, hence defining an equivalence of categories
%. The conjectured equivalence (\ref{geometric langlands}) is expected to be compatible with the fibration on the moduli space, in the sense that it is expected to restrict to an equivalence of categories 
\begin{equation} \label{fibre-wise geometric langlands}
D^b(\mathfrak{h}_G^{-1}(b)) \cong D^b(\mathfrak{h}_{^LG}^{-1}(b))
\end{equation}
for every $b \in B_G \cong B_{^LG}$. 

After the passage to the moduli scheme $\Mm_G$ of semistable Higgs bundles and the corresponding Hitchin fibration $h_G : \Mm_G \to B_G$, it is natural to ask if \eqref{fibre-wise geometric langlands} descends to a derived equivalence between the fibres $h^{-1}_G(b)$ and $h^{-1}_{^LG}(b)$. This has been obtained over the open locus $B_{reg} \subset B$ of smooth spectral curves - first by Hausel--Thaddeus \cite{hausel&thaddeus} for the Langlands dual pair $\SL_n$ and $\PGL_n$, and then Donagi--Pantev \cite{donagi&pantev} for a general Langlands dual pair $G$ and $^LG$. More specifically, Donagi--Pantev show that $h_G^{-1}(b)$ and $h_{^LG}^{-1}(b)$ are torsors for a pair of dual abelian varieties $A$ and $A^\vee$. Therefore, for $b \in B_{reg}$, a Fourier-Mukai transform 
\begin{equation} \label{fibre-wise geometric langlands in moduli}
D^b(h_G^{-1}(b)) \xrightarrow{\cong} D^b(h_{^LG}^{-1}(b))
\end{equation}
is provided by Mukai's equivalence \eqref{eq mukai}. 

For the Langlands self-dual group $G=\GL_n$, the aforementioned work of Arinkin \cite{arinkin1} initiated the extension of these ideas outside of $B_{reg}$. To a point $b \in B_{GL_n}$, the spectral correspondence \cite{beauville&arnaud&ramanan, schaub} assigns a \textit{spectral cover} $\beta_b : C_b \to X$, and when $C_b$ is reduced, the fibres are the compactified Jacobians ${h_{\GL_n}^{-1}(b) = \ol \Jac_{C_b}(q)}$. The Fourier--Mukai transform of Arinkin (\ref{eq FM on compactified Jac}) therefore implies \eqref{fibre-wise geometric langlands in moduli} for $\GL_n$ inside an intermediate locus $B_{reg} \subset B_{int} \subset B$ consisting of integral spectral curves. When the rank and the degree are coprime, the work of Melo-Rapagnetta-Viviani \cite{melo0, melo1, melo2} extends Arinkin's result to a larger locus $B_{red} \subset B$ consisting of reduced spectral curves. 

Our work has been driven by the natural challenge of extending the previous results to the framework associated to the next simplest pair of Langlands dual groups: $\SL_n$ and $\PGL_n$. Since $\PGL_n = \SL_n / \ZZ_n$, it follows that the moduli space $\Mm_{\PGL_n}$ can be identified with the finite quotient $\Mm_{\SL_n}/\Gamma$, with $\Gamma$ being the subgroup of $n$-torsion line bundles on $X$. Above a point $b \in B_{SL_n} \cong B_{PGL_n}$, corresponding to the spectral cover $\beta_b : C_b \to X$, the fibres of the Hitchin fibration are given by the compactified Prym variety ${h_{\SL_n}^{-1}(b) = \ol \Prym_{\beta_b}(q)}$ and the quotient ${h_{\PGL_n}^{-1}(b) = \ol \Prym_{\beta_b}(q)} / \Gamma$. \eqref{fibre-wise geometric langlands in moduli} conjectures these fibres are derived equivalent, proven in the regular locus by \eqref{eq FM for Prym}, and extended beyond the regular locus in this paper. 

Switching from the perspective of SYZ mirror symmetry to the homological mirror symmetry of Kontsevich, another motivation comes from the study of branes on the Higgs moduli spaces. $\Mm_G$ is naturally hyperk\"ahler, and Kapustin--Witten \cite{kapustin&witten} have introduced branes that are adapted to the hyperk\"ahler structure. They define a large class of $\BBB$-branes to be hyperholomorphic subvarieties equipped with a hyperholomorphic sheaf, and a large class of $\BAA$-branes to be holomorphic Lagrangian subvarieties equipped with a flat bundle. Kapustin--Witten then predict an exchange of $\BBB$ and $\BAA$-branes by a Fourier-Mukai transform. 

In light of this conjecture, many authors have studied examples of $\BBB$ and $\BAA$-branes under the Fourier--Mukai transform \eqref{fibre-wise geometric langlands in moduli} over $B_{reg}$, from which we highlight the work of Hitchin \cite{hitchin-charclasses}, Hausel--Hitchin \cite{hausel&hitchin} and the first author in collaboration with Jardim \cite{franco&jardim}. Restricting to the case $G = \GL_n$, recent progress has begun to probe the behaviour of branes supported outside of $B_{reg}$, such as \cite{franco&peon} and \cite{FGOP}, where the Fourier--Mukai transforms of Arinkin \cite{arinkin} and Melo--Rapagnetta--Viviani \cite{melo1, melo2} are crucial. One would like to make similar progress for $G = \SL_n$, in which case a corresponding Fourier-Mukai transform is essential.

\subsection{Our work}
Consider a curve $C$ and a finite flat morphism $\beta:C \to X$ with non-empty ramification. Take the pullback of the Poincar\'e sheaf $\Pp$ on $\ol \Jac_C(q) \times \ol \Jac_C(q')$ to define a sheaf $\Rr$ on $\ol \Prym_\beta(q) \times \ol \Prym_\beta(q')$. The main result of this article is that taking $\Rr$ as an integral kernel defines a derived equivalence, and thus a Fourier--Mukai transform. 
\begin{theorem_A} \label{theorem_a} (Theorem \ref{tm fourier-mukai general case}).
Suppose that $X$ is a smooth curve, and that $C$ is a projective curve that is reduced, connected, with locally planar singularities, and equipped with a finite flat morphism $\beta:C \to X$. With two generic polarisations $q$ and $q'$, denote the corresponding Prym varieties by $\ol \Prym_\beta(q)$ and $\ol \Prym_\beta(q')$. The integral functor with kernel $\Rr$,  
\begin{equation} \label{theorem_a equation}
\Psi^{\Rr}_{\ol \Prym \to \ol \Prym'}: D^b(\ol \Prym_\beta(q)) \to D^b(\ol \Prym_\beta(q'),{\Gamma}),
\end{equation}
is an equivalence of categories. 
\end{theorem_A}

We can say more once an endoscopy condition \cite{ngo, hausel&pauly, frenkel&witten} is taken into account. Recall that $\Gamma = \Jac_X[n]$ is the group of n-torsion line bundles on $X$. We say that $\beta: C \to X$ is {\it non-endoscopic} if the tensoral action of $\Gamma$ on the compactified Prym variety $\ol\Prym_\beta(q)$ is free, and $\beta: C \to X$ is {\it endoscopic} otherwise. 

In the case where $\beta : C \to X$ is non-endoscopic, it is possible to descend Theorem A to a Fourier-Mukai transform on the quotient. $\Rr$ descends to define a sheaf $\Qq$ on $\ol \Prym_\beta(q) \times \ol \Prym_\beta(q') / \Gamma$, and taking $\Qq$ as an integral kernel defines a derived equivalence between $\ol\Prym_\beta(q)$ and $\ol \Prym_\beta(q') / \Gamma$, and thus a second Fourier--Mukai transform

\begin{theorem_B} \label{theorem_b}
(Theorem \ref{tm main theorem}).
Suppose that $C$ is projective, reduced, connected with locally planar singularities, and equipped with a finite flat $n:1$ morphism $\beta : C \to X$ with non-empty ramification, such that $\Gamma$ acts freely on $\ol\Prym_\beta(q')$. There exists a sheaf $\Qq$ over $\ol{\Prym}_\beta(q) \times \ol{\Prym}_\beta(q')/\Gamma$ satisfying 
\begin{equation} \label{descent}
\left . \Pp \right|_{\ol{\Prym}_\beta(q) \times \ol{\Prym}_\beta(q')} \cong ( \id_{\ol{\Prym}} \times \pi)^{*} \Qq,
\end{equation}
where $\Pp$ is the Poincar\'e sheaf on $\ol \Jac_C(q) \times \ol \Jac_C(q')$ and $\pi$ denotes the quotient $\ol \Prym_\beta(q') \to \ol \Prym_\beta(q')/ \Gamma$.

Furthermore, the associated integral functor 
\begin{equation} \label{eq FM on compactified Prym}
\Phi^{\Qq}_{\ol\Prym \to \ol\Prym'/\Gamma}: D^b(\ol \Prym_\beta(q)) \to D^b(\ol \Prym_{\beta}(q') / \Gamma),
\end{equation}
is an equivalence of categories. 
\end{theorem_B}
We go on to comment on the relationship between Theorem A and B at the end of Section \ref{descent and endoscopy}, captured by a semi-orthogonal decomposition of $D^b(\ol \Prym_\beta(q'), \Gamma)$ that measures when a complex descends to the quotient. 

\subsection{Conclusions}
We take this opportunity to sketch applications of our results on the Higgs moduli spaces and the ideas presented in Section \ref{motivation}. Recall that for $G=\SL_n$ and $^LG=\PGL_n$, the spectral correspondence describes the fibres of the Hitchin fibration as compactified Prym varieties $h_{\SL_n}^{-1}(b) = \ol \Prym_\beta(q)$ and finite quotients $h_{\PGL_n}^{-1}(b) = h_{\SL_n}^{-1}(b) / \Gamma $. We define the {\it endoscopic locus} $B_{end} \subset B$ as those $b \in B$ for which the corresponding $n:1$-map $\beta_b: C_b \to X$ is endoscopic. This can also be characterised as the locus of spectral curves that are normalised by unramified covers of the base curve $X$. 

In the non-endoscopic case, the equivalence of categories in Theorem B, \eqref{eq FM on compactified Prym}, verifies the conjectural equivalence displayed in \eqref{fibre-wise geometric langlands in moduli} for points outside of the regular locus, extending the SYZ mirror symmetry picture to certain singular fibres. However for $b \in B_{end}$, the endoscopic case in Theorem A demonstrates that \eqref{fibre-wise geometric langlands in moduli} fails, precisely because some complexes in $D^b( \ol \Prym_{\beta}(q) , \Gamma)$ do not descend. Instead Theorem A should be interpreted in terms of a stack-theoretic quotient $[h_{\SL_n}^{-1}(b)/\Gamma]$ and an equivalence of categories
\[
D^b(h_{\SL_n}^{-1}(b)) \cong D^b([h_{\SL_n}^{-1}(b)/\Gamma]).
\]
This is an explicit demonstration of the well-known facts that (a) the geometric Langlands correspondence is incompatible with passage to the moduli spaces, and (b) SYZ mirror symmetry behaves in subtle ways on singular fibres.  

This observation becomes relevant when studying $\Gamma$-equivariant objects which do not descend to the quotient, as addressed by Frenkel--Witten \cite{frenkel&witten} in their study of fractional branes over $\Mm_{\SL_2)}$ and $\Mm_{\PGL_2}$. It also demonstrates that zerobranes $x \in \Mm_{\SL_n}$ that are supported on the endoscopic locus have mirror branes 
\[
^L\Oo_x = \Psi^{\Rr}_{\ol \Prym \to \ol \Prym}(\Oo_x) \in D^b([h_{\SL_n}^{-1}(b)/\Gamma]),
\]
and so these mirror branes are naturally stack-theoretic. In the language of Kapustin--Witten \cite{kapustin&witten}, $\Oo_x$ is an example of an electric eigenbrane, and $^L\Oo_x$ is a magnetic eigenbrane. This highlights the importance of a considering stack theoretic branes in mirror symmetry. 

\subsection{Structure}
The paper is organised as follows. Section \ref{subsec: FM and compactified Jac} recalls properties of compactified Jacobians proven in \cite{melo0} and the construction of the Poincar\'e sheaf due to Arinkin \cite{arinkin} and Melo--Rapagnetta--Viviani \cite{melo1, melo2}. 

Section \ref{sc Pryms} defines compactified Prym varieties following the work of Carbone \cite{carbone}. We shall observe that essential geometry geometric properties from Section \ref{subsec: FM and compactified Jac} restricts to the compactified Prym varieties.

Section \ref{sc constructing a kernel} studies the restriction of the Poincar\'e sheaf to the compactified Prym varieties - a sheaf $\Rr$ on $\ol \Prym_\beta(q) \times \ol \Prym_\beta(q')$. We provide properties including acyclicity, flatness and $\Gamma$-equivariance. 

Section \ref{equivalence equivariant derived category} deals with the construction of the Fourier-Mukai transform with kernel $\Rr$ to the equivariant derived category. This contains Theorem A, which establishes the equivalence of categories. The proof applies criteria of Bridgeland \cite{bridgeland} for fully-faithfulness and for equivalence. The later relies on the existence of left and right adjoints to $\Phi^\Qq_{\ol\Prym \to \ol\Prym/\Gamma}$ which are provided by the work of Hern\'andez-Ruip\'erez--Lopez-Mart\'{\i}n--Sancho-de-Salas \cite{HLS} on integral functors over Gorenstein schemes. 

Section \ref{FM to quotient} constructs the sheaf $\Qq$ on $\ol \Prym_\beta(q) \times \ol \Prym_\beta(q') / \Gamma$ by descent of $\Rr$ along the quotient map $\pi: \ol \Prym_\beta(q') \to \ol \Prym_\beta(q')/ \Gamma$. We show the integral functor \eqref{eq FM on compactified Prym} associated to $\Qq$ is an equivalence of categories, which provides us with Theorem B. We conclude by contrasting the endoscopic and non-endoscopic cases more explicitly, making use of a semi-orthogonal decomposition of $D^b(\ol \Prym_\beta(q'), \Gamma)$. 

\

\par From now on, \textit{scheme} shall always refer to a scheme of finite type over an algebraically closed field $k$ of characteristic $0$.

\subsection{Acknowledgements}

The first author wants to thank Margarida Melo and Filipo Viviani for enlightening conversations over previous versions of this work. Particular thanks are due to Filipo Viviani for suggesting the study of the equivariant derived category.

\

\par After the publication of a draft of a previous version of this work, the authors became aware that Michael Groechenig and Shiyu Shen have independently obtained similar results in \cite{groechenig&shen}.

\section{Fourier--Mukai transform on compactified Jacobian varieties} \label{subsec: FM and compactified Jac}

After Grothendieck, the Picard moduli functor of a projective $k$-scheme $Z$ is represented by a group scheme $\Pic(Z)$. We denote by $\Pic^0(Z)$ the connected component containing $\Oo_Z$. Altman and Kleiman \cite{altman&kleiman} generalized this construction showing that there exists an algebraic space $\Pic^{=}(Z)$ parametrizing simple torsion-free sheaves with rank $1$ on each irreducible component of $Z$. If $\Oo_Z$ itself is torsion-free, $\Pic(Z)$ is contained in $\Pic^{=}(Z)$ \cite{altman&kleiman} and we shall denote by $\ol\Pic^{\, 0}(Z)$ the connected component of $\Pic^{=}(Z)$ containing $\Oo_Z$.

When $Z=C$ is a connected projective curve, we call $\Jac_C = \Pic^{0}(C)$ the Jacobian variety of $C$. Furthermore, if $C$ is smooth, then $\Jac_C$ is an abelian variety of dimension equal to the genus of $C$. For an arbitrary abelian variety $A$, we call $\Pic^{0}(A)$ the dual abelian variety of $A$, which we denote by $\widehat A$. We recall that the Jacobian of a smooth projective curve is self-dual, $\widehat\Jac_C \cong \Jac_C$, while for general abelian varieties $A$ and $\widehat A$ are simply isogeneous. The universal family for $\Pic^0(A)$ is a line bundle on $A\times \widehat A$, called the Poincar\'e bundle. Using this bundle as an integral kernel, Mukai \cite{mukai} constructed a categorical equivalence between $D^b(A)$ and $D^b(\widehat{A})$, known as a Fourier--Mukai transform. 

If $C$ is singular, $\Jac_C$ is not projective in general, and one has at hand a broad literature on the compactifications \cite{altman&kleiman, esteves, simpson, melo0}. We shall assume that
\begin{equation*} 
\label{dagger} \text{$C$ is a projective, reduced, connected curve with locally planar singularities\tag{$\dagger$}}, 
\end{equation*}
and denote by $g$ its arithmetic genus. If $\{ C_i\}_{i \in \Irr_C}$ is the set of irreducible components of $C$, a {\it polarization} on $C$ is $q = \{q_i \}_{i \in \Irr_C}$, where $q_i \in \QQ$ and $|q|:= \sum_{i \in \Irr_C} q_i \in \ZZ$. We say that a polarization $q$ on $C$ is {\it general} if, for every non-trivial and connected subcurve $S \subset C$ whose complement is connected, one has that $|q|_S:= \sum_{i \in \Irr_{S}} q_i \notin \ZZ$. Given a torsion-free sheaf $\Ff$ on $C$ with $\chi(\Ff) = |q|$ and rank $1$ (on each irreducible component), we shall say that $\Ff$ is {\it $q$-stable} (resp. {\it $q$-semistable}), if for every proper subcurve $S \subset C$, one has 
\[
\chi(\Ff_S) > |q|_S \qquad (\textnormal{resp. } \chi(\Ff_S) \geq |q|_S),
\]
where $\Ff_S$ is the biggest torsion-free quotient of the restriction $\Ff|_S$. 

For our choice of $C$, Altman and Kleiman \cite{altman&kleiman} showed that $\Pic^{=}(C)$ is a scheme locally of finite type. or a general polarization $q$, Melo-Rapagnetta-Viviani \cite{melo0} and Esteves \cite{esteves} define the {\it fine compactified Jacobian} $\ol\Jac_C(q)$ to be the subscheme of $\Pic^{=}(C)$ parametrizing $q$-semistable sheaves in the connected component of $\Oo_C$. Esteves also shows that $\Pic^=(C)$ can be covered with $\ol\Jac_C(q)$ for varying choices of $q$. Since $q$ is general, $q$-semistability is equivalent to $q$-stability, and $\ol\Jac_C(q)$ is a fine moduli space. We denote by $\Uu_q$ the universal sheaf for the classification of $q$-stable pure dimension $1$ sheaves on $C$. 

\begin{theorem}[Theorem A \cite{melo0}]
\label{tm fine comp Jacobians}
Let $C$ be a curve as in \eqref{dagger} and $q$ a general polarization on $C$. Then, the fine compactified Jacobian $\ol\Jac_C(q)$ 
\begin{itemize}
\item is a connected reduced projective scheme with locally complete intersection singularities;
\item has trivial dualizing sheaf;
\item has as smooth locus the open subset $\Jac_C(q)$
parametrizing those line bundles which are $q$-stable. $\Jac_C(q)$ is a disjoint union of copies of $\Jac_C$.
\end{itemize}
\end{theorem}

Mukai's pioneering ideas in \cite{mukai} provided an autoequivalence of the derived category of the Jacobian of a smooth curve, the so-called Fourier--Mukai transform. Such derived equivalence was generalised by Arinkin \cite{arinkin} to a derived autoequivalence of compactified Jacobians over an integral planar curve (which are fine), and by Melo--Rapagnetta--Viviani \cite{melo0, melo1, melo2} to a derived equivalence between (possibly equal) fine compactified Jacobians, $\ol\Jac_C(q)$ and $\ol\Jac_C(q')$, associated to two general polarizations $q$ and $q'$ on a curve of the form \eqref{dagger}. The equivalence is provided by an integral functor having the so-called Poincar\'e sheaf as an integral kernel. Let us review the construction of this sheaf.

Consider universal families $\Uu$ and $\Uu^{0}$ for the fine moduli spaces $\ol \Jac_C (q)$ and $\Jac_C$, respectively. Denote the canonical projections $f_{ij}$, $f_i$ onto the corresponding factors, as depicted below.
\begin{equation}\label{eq: f projections}
\xymatrix{
&  C \times \ol{\Jac}_C(q) \times \Jac_C(q')  \ar[ld]_{f_{12}} \ar[d]_{f_{13}} \ar[rd]^{f_{23}} &
\\
C \times \ol{\Jac}_C(q) & C\times \Jac_C(q') & \ol{\Jac}_C(q) \times \Jac_C(q')
}
\quad \quad
\xymatrix{
&  C \times \ol{\Jac}_C(q)  \ar[ld]_{f_{1}} \ar[rd]^{f_{2}} &
\\
C & & \ol{\Jac}_C(q)
}
\end{equation}
Given a flat and proper morphism $f:Y\to S$ whose geometric fibres are curves, the determinant of cohomology defined in \cite{knudsen&mumford}, \cite[Section 6.1]{esteves} is a covariant functor 
\[ 
\Dd_f: \{ S\text{-flat sheaves on }Y \} \longrightarrow \{ \text{invertible sheaves on }S \}
\]
satisfying base change, projection and additivity formulae \cite[Propn. 44]{esteves}. Define the Poincar\'e line bundle \cite{melo1} over $\ol \Jac_C (q) \times \Jac_C(q')$ to be
\begin{equation}\label{eq:Poincarebundle}
\Pp^{'}:= \Dd_{f_{23}} \left ( f_{12}^*\Uu_q \otimes f_{13}^* \Uu^0_{q'}  \right )^{-1} \otimes \Dd_{f_{23}} \left ( f_{13}^* \Uu^0_{q'}  \right ) \otimes \Dd_{f_{23}} \left ( f_{12}^* \Uu_q  \right ).
\end{equation}
By base change with respect to Cartesian squares involving $f_{13}$ and $f_{12}$, one can check that the line bundle $\Pp'_M := \Pp' |_{\ol \Jac_C (q) \times \{M\}}$ is given by
\begin{equation} \label{eq description of Pp_M}
\Pp'_M = \Dd_{f_2} (\Uu_q \otimes f_1^*M)^{-1} \otimes \Dd_{f_2}(f_1^*M) \otimes \Dd_{f_2}(\Uu_q).
\end{equation}
%$\Uu$ is universal modulo the equivalence relation used in the moduli problem at hand - that means two class representatives $\Uu$ and $\Uu^{'}$ satisfy $\Uu \cong \Uu' \otimes f_2^{*} N$ for some $N \in Pic(\ol \Jac_X)$. By the projection formula the restriction (\ref{{eq description of Pp_M}}) is independent of the choice of representative. 
Repeat the construction symmetrically to obtain a Poincar\'e line bundle $\Pp^{\prime\prime}$ over $\Jac_C(q) \times \ol \Jac_C (q')$. By symmetry, they agree over the common open subset $\Jac_C(q) \times \Jac_C(q')$, and so $\Pp'$ and $\Pp^{\prime\prime}$ glue to define a line bundle $\Pp^{\sharp}$ over the union of their supports,
\[
\left (\ol{\Jac}_C(q) \times \ol{\Jac}_C(q') \right)^\sharp 
:= \left( \Jac_C(q) \times \ol{\Jac}_C(q') \right) \cup \left( \ol{\Jac}_C(q) \times \Jac_C(q') \right),
\]
whose complement has codimension $2$. Denote by $\xi$ the inclusion morphism into $\ol{\Jac}_C(q) \times \ol{\Jac}_C(q')$. We define the {\it Poincar\'e sheaf} to be the pushforward 
\begin{equation} \label{eq def Poincare sheaf}
\Pp := \xi_{*} \Pp^{\sharp}.
\end{equation}

A key feature of the Poincar\'e sheaf is that it is a maximal Cohen-Macaulay sheaf, flat over each factor \cite{arinkin, melo2}. This is proven by showing that $\Pp$ is the descent of a certain maximal Cohen--Macaulay sheaf $\Gg$ on $\Hilb_C \times \ol \Jac_C(q')$ via the Abel--Jacobi map.

With $\Pp$ as the kernel we can write down the Fourier--Mukai transform. Define the projections $\p_1$, $\p_2$ onto the first and second factors
\begin{equation}\label{eq:projections-compact-Jac}
\xymatrix{
&  \ol{\Jac}_C(q) \times \ol{\Jac}_C(q')  \ar[ld]_{\p_1} \ar[rd]^{\p_2} &
\\
\ol{\Jac}_C(q) & & \ol{\Jac}_C(q'). 
}
\end{equation}
and consider the integral functor 
\begin{equation} \label{eq FM Arinkin}
\morph{D^b \left ( \ol{\Jac}_C(q) \right )}{D^b \left ( \ol{\Jac}_C(q') \right )}{\Ff^\bullet}{R \p_{2,*}(\p_1^*\Ff^\bullet \otimes \Pp).}{}{\Phi^{\Pp}_{\ol\Jac \to \ol\Jac'}}
\end{equation}
Extending the work of Arinkin \cite{arinkin} for integral curves, Melo–Rapagnetta–Viviani \cite{melo2} prove the following.
\begin{theorem} \label{thm:derivedequiv}
Let $C$ be curve satisfying \eqref{dagger} and let $q$ and $q'$ be two general polarisations on $C$ such that $|q| = |q'| = \chi(\Oo_C)$. Then, the above integral functor is a derived equivalence.
\end{theorem}
The condition $|q| = |q'| = \chi(\Oo_C)$ is not strictly necessary, but we assume this in order to make the choice of $\Pp$ natural. Observe that, for every $q$ and $q'$ such that $|q| = |q'|$, one has $\Jac_C(q) = \Jac_C(q')$, which we abbreviate by $\Jac_C$ when further $|q| = |q'| = \chi(\Oo_C)$.

Using this derived equivalence, Arinkin showed that compactified Jacobians of integral curves are self-dual \cite{arinkin} and Melo--Viviani--Rapagnetta extended this statement \cite[Thm. B]{melo2} to the case of singular curves of the form \eqref{dagger}. They showed that the morphism
\[
\morph{\ol\Jac_C(q')}{\ol\Pic^{\, 0}(\ol\Jac_C(q))}{\Nn}{\Pp|_\Nn}{}{\Theta_C}
\]
is an open embedding. It is conjectured to be an isomorphism for any reduced curve with locally planar singularities. One sees that $\Theta_C$ restricts to an isomorphism of groups 
\begin{equation} \label{eq Psi^0}
\Theta_C^0 : \Jac_C \stackrel{\cong}{\longrightarrow}  \Pic^0(\ol\Jac_C(q)),
\end{equation}
which is equivariant with respect to the action of $\Jac_C$ on $\ol\Jac_C(q')$ and $\Pic^0(\ol\Jac_C(q))$ on $\ol\Pic^{\, 0}(\ol\Jac_C(q))$ under tensorization. We observe that $\Theta_C^0$ depends only on the construction of the Poincar\'e bundle $\Pp' \to \ol\Jac_C(q) \times \Jac_C$ obtained in \cite{melo1}.

\section{Compactified Prym varieties}
\label{sc Pryms}

We now introduce compactified Prym varieties associated to a curve $C$ satisfying \eqref{dagger}, a smooth curve $X$ and a finite morphism $\beta: C \to X$ with non-empty ramification. We also take a general polarization $q$ such that $|q| = \chi(\Oo_C)$, so the compactified Jacobian $\ol\Jac_C(q)$ is fine. Consider the map 
\[
\Nm_\beta := \det(R\beta_*\Oo_C)^{-1} \otimes \det \circ R\beta_* : \ol\Jac_C(q) \longrightarrow \Jac_X.
\]
Following Carbone \cite{carbone}, the \textit{compactified Prym variety} is defined to be the subvariety
\[
\ol{\Prym}_\beta(q) := \ker \Nm_\beta.
\]
We consider as well $\Prym_\beta(q)$ to be the locally-free locus of $\ker \Nm_\beta$, {\it i.e.} its intersection with $\Jac_C(q)$. Thanks to Hausel--Pauly \cite{hausel&pauly}, the restriction of $\Nm_\beta$ to $\Jac_C(q)$ coincides with the Norm map. This justifies calling $\Prym_\beta(q)$ and $\ol \Prym_\beta(q)$ the Prym variety and the compactified Prym variety associated to $\beta$ and $q$. 

Observe that $\beta^*\Jac_X$ is naturally a subgroup of $\Jac_C$ isomorphic to $\Jac_X$, as $\beta$ is ramified and $\ker \beta^*$ is just $\Oo_X$. One has a natural action of $\Jac_C$ over $\overline{\Jac}_C(q)$ by tensorization, and so does its subgroup $\beta^*\Jac_X$. The projection formula ensures for any $\Ff \in \ol{\Jac}_C(q)$ and any $L \in \Jac_X$,
\begin{equation} \label{eq Nm and otimes}
\Nm_\beta \left (  \beta^*L \otimes \Ff \right ) \cong L^n \otimes \Nm_\beta(\Ff),
\end{equation}
as was noted by Carbone \cite{carbone}. Hence, for every $\Ff \in \ol\Prym_\beta(q)$, the composition 
\[
\Jac_X \stackrel{\Ff \otimes \beta^*}{\longrightarrow} \ol\Jac_C(q) \stackrel{\Nm_\beta}{\longrightarrow} \Jac_X
\]
does not depend on the choice of $\Ff$ and amounts to the surjection $[n]:\Jac_X \to \Jac_X$, $L \mapsto L^{\otimes n}$, whose kernel is the subgroup of $n$-torsion line bundles that we shall denote by
\[
\Gamma := \Jac_X[n].
\] 
Furthermore, $\Gamma$ is the subgroup of $\beta^*\Jac_X$ that preserves $\ol{\Prym}_\beta(q)$ and any $\beta^*\Jac_X$-orbit intersects $\ol{\Prym}_\beta(q)$ in a $\Gamma$-orbit, as one can always choose $L$ to be $\Nm_\beta(\Ff)^{-1/n}$ and this choice is unique up to $\Gamma$. This provides the following description that can be found in \cite{hausel&thaddeus} for the case of $C$ smooth.

%\par \textcolor{blue}{(J: Instead of $-1/n$ it should be $-n$ right? Also, the reference to $L$ is a bit confusing when we say the orbits intersect. We should also mention that $n$ should be the degree of the ramified cover $\beta$)}.

\begin{lemma}
	\label{lm decomposition of olJac_C}
Let $C$ be a curve as in \eqref{dagger}, $X$ a smooth curve, $q$ a general polarization on $C$ and $\beta : C \to X$ a finite flat morphism with non-empty ramification. One has the isomorphism of schemes 
	\begin{equation} \label{eq description of J_C in terms of Prym}
	\overline{\Jac}_C(q) \cong \quotient{\Jac_X \times \, \ol{\Prym}_\beta(q)}{\Gamma}. 
	\end{equation}
	Hence, the composition of $[n]$ with the projection to the first term 
	\[
	\overline{\Jac}_C(q) \longrightarrow \quotient{\Jac_X}{\Gamma},
	\]
	amounts to the composition of $\Nm_\beta$ with the quotient map $\Jac_X \to \quotient{\Jac_X}{\Gamma}$. Moreover, the projection to the second term is a map
	\begin{equation} \label{eq chi}
	\chi : \overline{\Jac}_C(q) \longrightarrow \quotient{\ol{\Prym}_\beta(q)}{\Gamma},
	\end{equation}
	with smooth fibres isogenous to $\Jac_X$.
\end{lemma}

\begin{proof}
Consider 
\[
\xymatrix{
& & \ol\Jac_C(q) \ar[d]^{\Nm_\beta} 
\\
\Jac_X \ar[rr]_{[n]} & & \Jac_X,
}
\]
where $[n]$ is given by considering the $n$-th power, and it is an \'etale morphism. Since \'etaleness is stable under base change, and $[n] :  \Jac_X \to \Jac_X$ is an \'etale morphism, one has that the associated fibre product is \'etale over $\ol\Jac_C(q)$. Furthermore, this fibre product is isomorphic $\Jac_X \times \ol\Prym_\beta(q)$, via
\[
\map{\Jac_X \times_{\Jac_X} \overline{\Jac}_C(q)}{\Jac_X \times \ol{\Prym}_\beta(q)}{(L,\Gg)}{(L, \beta^*L^{-1} \otimes \Gg).}{}
\]
Considering the composition of both maps, one gets the \'etale morphism
\begin{equation} \label{eq etale map over olJac}
\morph{\Jac_X \times \ol{\Prym}_\beta(q)}{\overline{\Jac}_C(q)}{(L,\Ff)}{\beta^*L \otimes \Ff.}{}{\eta}
\end{equation}
As the action of $\Gamma$ is faithful on $\Jac_X$, it acts faithfully on the source, so $\eta$ is which is an \'etale $\Gamma$-covering after \eqref{eq Nm and otimes}. Hence, $\eta$ factors through the quotient proving the first statement. The second statement follows easily from this description.
\end{proof}

This allow us to obtain an analogous description to that of Theorem \ref{tm fine comp Jacobians}.

\begin{corollary} \label{co description of fine comp Pryms}
Let $C$ be a curve as in \eqref{dagger}, $X$ a smooth curve, $q$ a general polarization on $C$ and $\beta : C \to X$ a finite flat morphism with non-empty ramification. Then $\ol\Prym_\beta(q)$ 
\begin{itemize}
\item is a connected reduced projective scheme of finite type with locally complete intersection singularities;
\item has trivial dualizing sheaf and is therefore Gorenstein;
\item has smooth locus $\Prym_\beta(q) = \ol\Prym_\beta(q)^{sm} \subset \ol\Prym_\beta(q)$ parametrizing line bundles on $C$.
\end{itemize}
\end{corollary}

\begin{proof}
The three statements follow from the existence of the \'etale map \eqref{eq etale map over olJac} and the corresponding statements of Theorem \ref{tm fine comp Jacobians}.
\end{proof}

The action of $\Jac_C$ on $\ol\Jac_C(q)$ produces, for all $L \in \Jac_C$, an automorphism of the compactified Jacobian. For any $L_\gamma \in \Gamma$, we denote the associated automorphism by
\[
\morph{\ol{\Jac}_C(q)}{\ol{\Jac}_C(q)}{\Ff}{\Ff \otimes L_\gamma.}{\cong}{\tau_\gamma}
\]
After \eqref{eq Nm and otimes}, tensoring by $L_\gamma$ preserves $\ol\Prym_\beta(q)$, the kernel of the Norm map. Then, $\tau_\gamma$ preserves $\ol \Prym_\beta(q)$, thus defining an automorphism of it. To study the fixed point set $\ol \Prym_\beta(q)^{\tau_\gamma}$, we first describe when it is non-empty. 

\begin{lemma} \label{lm fixed points Jac and Prym}
Under the same hypothesis of Lemma \ref{lm decomposition of olJac_C} and Corollary \ref{co description of fine comp Pryms}, 
\[
\ol \Prym_\beta(q)^{\tau_\gamma} = \ol \Jac_C(q)^{\tau_\gamma} \cap \ol \Prym_\beta(q).
\]
%the fixed point set $\ol \Prym_\beta(q)^{\tau_\gamma}$ is empty if and only if the fixed point set $\ol \Jac_C(q)^{\tau_\gamma}$ is.
\end{lemma}

\begin{proof}
The proof is straight-forward.
%We just need to check the $\supset$ statement. Using \eqref{eq description of J_C in terms of Prym}, suppose that $[(L,\Ff)]_\Gamma$ belongs to $\ol \Jac_C(q)^{\tau_\gamma}$. Since $\Gamma$ acts faithfully on $\beta^*\Jac_X$, one has that $\Ff$ is forcely fixed by $\tau_\gamma$, giving a point in $\ol \Prym_\beta(q)^{\tau_\gamma}$.
\end{proof}

In order to describe the fixed points of $\tau_\gamma$, recall that, associated to every $L_\gamma \in \Gamma = \Jac_X[n]$ of order $n$, there exists an $n:1$ \'etale map $\pi : \Sigma_\gamma \to X$ with Galois group $\langle L_\gamma \rangle$. 

\begin{proposition} \label{pr fixed points of tau_gamma}
Let $C$ be a curve as in \eqref{dagger} and denote by $C_1, \dots, C_n$ its irreducible components. Suppose that our curve is equipped with a finite flat $n:1$ morphism $\beta : C \to X$ with ramifitacion, inducing a ramified cover $\beta_i : C_i \to X$ of order $\ell_i$. Consider a non-trivial element $L_\gamma \in \Gamma$ and the associated automorphism $\tau_\gamma$ of $\ol\Jac_C(q)$. Then either:
\begin{itemize}
    \item the fixed point set $\ol{\Jac}_C(q)^{\, \tau_\gamma}$ is empty, or: 
    \item each $C_i$ is normalized by $\Sigma_{\gamma_i}$, for some $\gamma_i \in \Jac_X[\ell_i]$, of order $\ell_i$, is such that $L_{\gamma} \in \langle L_{\gamma_i} \rangle$.
\end{itemize}

In the non-empty case, the fixed point set $\ol\Jac_C(q)^{\tau_\gamma}$ corresponds to those sheaves obtained by push--forward under a generically $1:1$ finite morphism $\nu : \wt C \to C$, where $\wt C$ is a curve (possibly non-connected) whose irreducible components are $\Sigma_{\gamma_i}$. Moreover $\wt \beta = \beta \circ \nu : \wt C \to X$ is an $n:1$ cover.
\end{proposition}

\begin{proof}
First, let us suppose that there exists $\wt C$ and $\wt \beta$ as in the hypothesis. Note that $L_\gamma \cong L_{\gamma'}^{\otimes t}$ for some $t$. This implies that the pull-back of $L_\gamma$ to $\Sigma_{\gamma'}$ is trivial, and so is its pull-back to $\wt C$. Therefore, as a direct consequence of projection formula and this fact, the sheaves obtained by pushing--forward from $\wt C$ are fixed by tensorization under $L_\gamma$.

Let us now show that all fixed points are obtained as pushforward under some $\nu$ as above. A fixed point under $\tau_\gamma$ is ($q$-stable) sheaf $\Ff \in \ol{\Jac}_C(q)$ for which there exists an isomorphism $\alpha : \Ff \stackrel{\cong}{\longrightarrow} \Ff \otimes \beta^*L_\gamma$. The pair $(\Ff,\alpha)$ gives rise to a sheaf of $\Oo_{\tot(\beta^*L_\gamma)}$-modules $\wt \Ff$, and denote its support by $\wt C \subset \tot(\beta^*L_\gamma)$, and by $\nu : \wt C \to C$ the projection naturally obtained from the structural morphism of $\beta^*L_\gamma$. %Since $\alpha$ is an isomorphism, none of its eigenvalues vanish, so $\wt C$ does not intersect the zero section of $\tot(\beta^*L_\gamma)$.
Since $\Ff$ is a rank $1$ torsion-free sheaf, $\wt C$ is generically $1:1$ onto $C$. It remains to be shown that the irreducible components of $\wt C$ are of the form $\Sigma_{\gamma_i}$. 
	
Given an irreducible component $C_i$ of $C$	we denote by $\hat C_i$ the subcurve obtained by the union of all the irreducible components of $C$ except $C_i$. We denote as well $\Ff_i$ to be the kernel of the projection $\Ff \to \Ff|_{\hat C_i}$. This is a torsion-free subsheaf of $\Ff$ supported on the irreducible component $C_i$. The isomorphism $\alpha$ induces $\alpha_i : \Ff_i \stackrel{\cong}{\to} \Ff_i \otimes \beta_i^* L_\gamma$, where we denote by $\beta_i : C_i \to X$ the restriction of $\beta$ to $C_i$. Consider as well the sheaf of $\Oo_{\tot(\beta_i^*L_\gamma)}$-modules $\wt \Ff_i$ obtained from $(\Ff_i, \alpha_i)$, and observe that $\wt \Ff_i$ is a subsheaf of $\wt \Ff$ supported on a subcurve $\wt C_i$ of $\wt C$, equipped with a generically $1:1$ morphism $\nu_i : \wt C_i \to C_i$. Hence, $\wt C_i$ is an irreducible component of $\wt{C}$ which is naturally equipped with the $\ell_i:1$ covering $\wt \beta_i = \beta_i \circ \nu_i : \wt C_i \to X$. The isomorphism $\alpha_i$ is obtained from pushing-forward the morphism $\lambda_i : \wt \Ff_i \to \wt \Ff_i \otimes \beta_i^*L_\gamma$ obtained by tensorization with the tautological section of $\beta_i^*L_\gamma$.

Since $\Ff_i$ is torsion-free, by projection formula, $\alpha_i : \Ff_i \stackrel{\cong}{\to} \Ff_i$ gives rise to the isomorphism $\beta_{i,*}\alpha_i : \beta_{i,*}\Ff_i \stackrel{\cong}{\to} \beta_{i,*} \Ff_i \otimes L_\gamma$, of rank $\ell_i$ vector bundles over $X$. Once more, the pair $(\beta_{i,*}\Ff_i, \beta_{i,*}\alpha_i)$ provides a $\Oo_{\tot(L_\gamma)}$-module $\Ff'_i$, whose support does not intersect the $0$-section $X \subset \tot(L_\gamma)$ as $\beta_{i,*}\alpha_i$ is an isomorphism. Then, the support of $\Ff'_i$ is an unramified $\ell_i:1$ cover of $X$ where the pull-back of $L_\gamma$ trivializes, {\it i.e.} $\Ff'_i$ is supported on $\Sigma_{\gamma_i}$, for some $\gamma_i \in \Jac_X[\ell_i]$ of order $\ell_i$ with $L_\gamma \in \langle L_{\gamma_i} \rangle$. 

Observe that the $\ell_i:1$ covering $\beta_i : C_i \to X$ lifts naturally to $\beta'_i : \tot(\beta_i^*L_\gamma) \to \tot(L_\gamma)$, which we denote with the same symbol by abuse of notation. Let us denote by $\pi_i$ the structural morphism $\tot(L_\gamma) \to X$ and observe that it gives rise to a commuting diagram since $\pi_i \circ \beta'_i = \beta_i \circ \nu_i$. Recall that $\Sigma_{\gamma_i} $ is determined by $\beta_{i,*} \alpha_i = \beta_{i,*} \nu_{i,*} \lambda_i = \pi_{i,*} \beta'_{i,*}\lambda_i$. It then follows that $\beta'_i$ sends $\wt C_i$ to $\Sigma_{\gamma_i}$. Both $\wt{C}_i$ and $\Sigma_{\gamma_i}$ are finite $\ell_i:1$-covers of $X$ and $\beta'_i: \wt{C}_i \to \Sigma_{\gamma_i}$ commutes with the coverings, it is generically $1:1$. Since $\Sigma_{\gamma_i}$ is smooth and irreducible, it then follows that $\wt C_i \cong \Sigma_{\gamma_i}$.   
\end{proof}

\begin{remark}
We recall from Hausel--Pauly \cite{hausel&pauly} that whenever an irreducible curve is normalised by an unramified cover, the Prym variety $\Prym_\beta$ is non-connected. Hence, $\ol\Prym_\beta$ is reducible even if $\ol \Jac_C$ is irreducible.
\end{remark}

In view of the above, and inspired by \cite{ngo}, \cite{frenkel&witten} and \cite{hausel&pauly}, we say that the finite flat $n:1$ morphism $\beta : C \to X$ is of {\it endoscopic type} if the $\ol \Prym_\beta(q)^{\tau_\gamma}$ is non-empty for some $\gamma \in \Gamma$. If such condition is not satisfied, we say that the curve $C$ is of {\it non-endoscopic type}. Note that curves on non-endoscopic type are generic.

We study the properties of quotient $\ol \Prym_\beta(q)/\Gamma$ in the last case.

\begin{proposition} \label{pr quotient prym}
Suppose that $C$ is a curve as in \eqref{dagger} equipped with a finite flat $n:1$ cover $\beta : C \to X$ with non-empty ramification and of non-endoscopic type. Consider a general polarization $q$ on $C$. Then
\begin{itemize}
\item the projection 
\[
\pi : \ol \Prym_\beta(q) \longrightarrow \ol \Prym_\beta(q)/\Gamma,
\]
is an \'etale morphism; 

\item $\ol \Prym_\beta(q)/\Gamma$ is a connected reduced projective scheme of finite type;
\item $\ol \Prym_\beta(q)/\Gamma$ is Gorenstein.
\end{itemize}
\end{proposition}

\begin{proof} 
If $\beta$ is non-endoscopic, the action of $\Gamma$ is free, so $\pi$ is \'etale. The second and third statements then follow naturally from Corollary \ref{co description of fine comp Pryms}. 
\end{proof}

\begin{remark}
In the non-endoscopic case, the projection $\chi$ is smooth. It is so since its fibres are smooth (Lemma \ref{lm decomposition of olJac_C}), and it is flat thanks to the commutativity of
\[
\xymatrix{
\Jac_X \times \ol\Prym_\beta(q) \ar[rr]^{\eta} \ar[d]_{p_2} & & \quotient{\Jac_X \times \, \ol{\Prym}_\beta(q)}{\Gamma} \ar[d]^\chi
\\
\ol{\Prym}_\beta(q) \ar[rr]_\pi & & \quotient{\ol{\Prym}_\beta(q)}{\Gamma},
}
\]
with $p_2$ being flat and the action of $\Gamma$ faithful. 
\end{remark}

\section{Fourier--Mukai transform on compactified Prym varieties} 
\label{equivariant derived category}

\subsection{The restriction of the Poincar\'e sheaf}
\label{sc constructing a kernel}

Pullback the Poincar\'e sheaf $\Pp$ on $\ol \Jac_C(q) \times \ol \Jac_C(q')$ defined in \eqref{eq def Poincare sheaf} along the embeddings $\jmath : \ol \Prym_\beta(q) \to \ol \Jac_C(q)$ and $\jmath' : \ol \Prym_\beta(q') \to \ol \Jac_C(q')$ to define the sheaf
\[
\Rr := (\jmath \times \jmath')^*\Pp.
\]
The next section takes $\Rr$ as an integral kernel to define a Fourier--Mukai transform, and this section is dedicated to the preparatory study of $\Rr$. 

Since Arinkin has proven that $\Pp$ is maximal Cohen--Macaulay, one can use homological criterion to show the following acyclicity property (see for instance \cite[Lma. 2.3]{arinkin}). 

\begin{lemma} \label{derived pullback is pullback} There is an isomorphism of sheaves
\[
(\jmath \times \jmath')^*\Pp \cong L(\jmath \times \jmath')^{*}\Pp, 
\]
so that $\Rr$ may be equivalently defined as the derived pullback. 
\end{lemma}
Next we explore the flatness of $\Rr$. 
\begin{lemma} \label{lm flatness of Rr}
The sheaf $\Rr$ over $\ol\Prym_\beta(q) \times \ol\Prym_\beta(q')$ is flat over each of the projections onto $\ol\Prym_\beta(q)$ and $\ol\Prym_\beta(q')$. 
\end{lemma}

\begin{proof}
$\Mm := (\jmath \times \id)^*\Pp$ is a sheaf over $V := \ol\Prym_\beta(q) \times \ol\Jac_C(q')$. Since $\Pp$ is flat over each of the factors, by base change $\Mm$ is flat over $\ol\Prym_\beta(q)$. Pick now a point $z \in W := \ol\Prym_\beta(q) \times \ol\Prym_\beta(q')$ and denote by $z_1 \in \ol\Prym_\beta(q)$ and $z_2 \in \ol\Prym_\beta(q')$ its image under the projection onto the first and second factors. For the duration of this proof we shall apply the contractions $\ol \Prym = \ol \Prym_{\beta}(q)$ and $\ol \Jac = \ol \Jac_C(q')$. We need to see that $\Rr_z$ is a free $\Oo_{\ol\Prym, z_1}$-module, provided $\Mm_z$ is. Observe that one has a projection $\Mm_z \to \Rr_z$, and complete it into a short exact sequence of $\Oo_{\ol\Prym, z_1}$-modules
\[
0 \longrightarrow \Ii \longrightarrow \Mm_z \longrightarrow \Rr_z \longrightarrow 0.
\]
One has that 
\[
\Oo_{V,z} \cong \Oo_{\ol\Prym, z_1} \times_k \Oo_{\ol\Jac , z_2}. 
\]
Since $\Oo_{V,z}$ can naturally be understood as a sheaf of endomorphisms of $\Mm_z$ and the later is $\Oo_{\ol\Prym, z_1}$-flat, one has the decomposition of $\Oo_{\ol\Prym,z_1}$-modules
\[
\Mm_z \cong \Oo_{\ol\Prym, z_1} \times_k M
\]
for some $\Oo_{\ol\Jac , z_2}$-module $M$. Consider $x_1, \dots, x_\ell$ to be local coordinates of $\Jac_X$ in a neighbourhood of the identity. Note that one has 
\[
k[x_1, \dots, x_\ell] \hookrightarrow \Oo_{\ol\Jac , z_2}
\]
associated to the projection $\Nm_\beta : \ol\Jac_C \to \Jac_X$ with $z_1$ mapping to the identity. One can then describe the submodule of $\Mm_z$ generated by the action of the ideal $\langle x_{1}, \dots, x_\ell \rangle$, 
\[
\Ii = \langle x_{1}, \dots, x_\ell \rangle \Mm_z.
\]
Furthermore, setting 
\[
N := \langle x_{1}, \dots, x_\ell \rangle M,
\]
one has 
\[
\Ii \cong  \Oo_{\ol\Prym,z_1} \times_k N.
\]
Then
\[
\Rr_z = \Oo_{\ol\Prym,z_1} \times_k M/N,
\]
is flat over $\Oo_{\ol\Prym, z_1}$ and we conclude the proof of flatness over the first factor. 

Considering $(\id \times \jmath')^*\Pp$, one shows mutatis mutandis that $\Rr$ is flat over the second factor too. 
\end{proof}

Following-up, we proof that $\Rr$ is Cohen--Macauly. Let us consider the open subset
\[
\left ( \ol{\Prym}_\beta(q) \times \ol{\Prym}_\beta(q) \right )^\sharp := \left ( \ol{\Prym}_\beta(q) \times \Prym_\beta \right ) \cup \left ( \Prym_\beta \times \ol{\Prym}_\beta(q) \right ),
\]
and denote the inclusion map by 
\[
\zeta : \left ( \ol{\Prym}_\beta(q) \times \ol{\Prym}_\beta(q) \right )^\sharp \into \ol{\Prym}_\beta(q) \times \ol{\Prym}_\beta(q).
\]
Observe that the restriction of $\Rr$ there,
\[
\Rr^\sharp := \zeta^*\Rr,
\]
is locally free as it can be obtained by restricting the Poincar\'e line bundle $\Pp^\sharp$. 

\begin{lemma} \label{lm Rr is CM}
$\Rr$ is a Cohen--Macaulay sheaf, satisfying
\begin{equation} \label{eq Rr as push-forward}
\Rr \cong \zeta_{*}  \Rr^\sharp.
\end{equation}
\end{lemma}

\begin{proof}
Equation \eqref{eq Rr as push-forward} follows from \eqref{eq def Poincare sheaf} and a series of flat base change theorems. Verdier duality ensures that $\zeta_*$ commutes with the dualizing functors. Then, as $\Rr^\sharp$ is a line bundle (hence, Cohen--Macaulay), \eqref{eq Rr as push-forward} implies that the derived dual of $\Rr$ is supported in single degree.
\end{proof}

Before continuing with our study of $\Rr$, we state some properties of its endomorphisms sheaf.  

\begin{lemma} \label{lm End Rr}
The endomorphisms complex of $\Rr$ is indeed the trivial sheaf, 
\[
R\Hhom(\Rr, \Rr) \cong \Hhom(\Rr, \Rr) \cong \Oo_{\ol\Prym \times \ol \Prym}.
\]
\end{lemma} 

\begin{proof}
Observe that $\ol \Prym (q) \times \ol\Prym (q')$ has trivial dualizing sheaf after Corollary \ref{co description of fine comp Pryms}. Then, 
\[
R\Hhom \left (\Rr, \Rr \right ) \cong R\Hhom \left (\Rr, \Rr \otimes \omega_{\ol \Prym \times \ol \Prym} \right )
\]
and, since $\Rr$ is Cohen--Macaulay after Lemma \ref{lm Rr is CM}, its endomorphism complex is supported in single degree,
\[
R\Hhom \left (\Rr, \Rr \right ) \cong \Hhom \left (\Rr, \Rr \right ).
\]
Furthermore, thanks to \eqref{eq Rr as push-forward}, the endomorphisms sheaf $\Hhom (\Rr, \Rr)$ is Cohen--Macaulay too, as 
\[
R\Hhom\left ( \Hhom \left (\Rr, \Rr \right ), \omega_{\ol\Prym \times \ol \Prym}[d] \right ) \cong R\Hhom \left (\Rr, \Rr \otimes \omega_{\ol \Prym \times \ol \Prym}[d] \right ) \cong R\Hhom \left (\Rr, \omega_{\ol \Prym \times \ol \Prym}[d] \right ) \otimes^L \Rr 
\]
is supported in single degree as so is $R\Hhom \left (\Rr, \omega_{\ol\Prym \times \ol \Prym}[d] \right )$. Since $\Rr^\sharp$ is invertible, local Verdier duality implies that the endomorphism sheaf of $\Rr$ is the pushforward of a line bundle,
\[
\Hhom(\Rr, \Rr) \cong \Hhom (\zeta_* \Rr^\sharp, \Rr) \cong \zeta_* \Hhom(\Rr^\sharp, \zeta^{!}\Rr) \cong \zeta_*\left ( (\Rr^{\sharp})^{-1} \otimes \zeta^{!}\Rr \right ) \cong \zeta_* \omega_{\zeta}. 
\]
Applying local Verdier duality once more, one observes
\[
\zeta_* \omega_{\zeta} \cong \zeta_* \Hhom (\Oo_{(\ol\Prym \times \ol \Prym)^\sharp}, \zeta^{!} \Oo_{\ol\Prym \times \ol \Prym}) \cong \Hhom \left ( \zeta_*\Oo_{(\ol\Prym \times \ol \Prym)^\sharp} , \Oo_{\ol\Prym \times \ol \Prym} \right ) \cong \Hhom \left ( \Oo_{\ol\Prym \times \ol \Prym},  \Oo_{\ol\Prym \times \ol \Prym}\right ),
\]
and the result follows. 
\end{proof}

In the next step we shall show that $\Rr$ is equivariant with respect to the action of $\Gamma$. Before that, we need the following proposition and its corollary. Denote by $\Pp_X \to \Jac_X \times \Jac_X$ the Poincar\'e bundle associated to the smooth and self-dual abelian variety $\Jac_X$. With the isomorphisms of groups $\Theta_X$ and $\Theta^0_C$ defined in \eqref{eq Psi^0} at hand, let us describe the image of $\beta^*\Jac_X$ under it. Let us recall the Poincar\'e line bundle $\Pp' \to \ol\Jac_C(q) \times \Jac_C$ from Section \ref{subsec: FM and compactified Jac}.

\begin{proposition} \label{pr Psi exchanges quotients}
Let $C$ be a curve as in \eqref{dagger}, $X$ a smooth curve, $q$ a general polarization on $C$, and $\beta : C \to X$ a finite flat morphism with non-empty ramification. Consider the embedding $i : \beta^* \Jac_X \hookrightarrow \Jac_C$. Then, one has
\[
(\id \times  i)^* \Pp' \cong (\Nm_\beta \times \id)^* \Pp_X,
\]
inducing 
\[
\Nm_{\beta}^* \, \circ \, \Theta_X = \Theta_C^0 \, \circ \, \beta^*.
\]
\end{proposition}

\begin{proof}
It is enough to prove the first statement as the second follows naturally from the observation that $\Theta_X$ and $\Theta_C^0$ are isomorphisms and $\Nm_\beta^* \widehat\Jac_X$ and $\beta^*\Jac_X$ are isogenous. 

Since $\Pp'$ is locally free, $(\id \times i)^* \Pp$ can be seen as a family of line bundles over $\Jac_X$, so by the universal property of the Poincar\'e bundle $\Pp_X$: 
\[
(\id \times i)^* \Pp' \cong (\psi \times \id)^* \Pp_X \otimes p_1^* \Ww,
\]
for some line bundle $\Ww$ on $\ol\Jac_C(q)$ and some morphism 
\[
\psi : \ol \Jac_C(q) \longrightarrow \Jac_X.
\]
We assume that both Poincar\'e bundles $\Pp'$ and $\Pp_X$ are normalized, {\it i.e.} $\Pp'|_{\ol\Jac_C \times \{ \Oo_C \}} \cong \Oo_{\ol\Jac_C}$ and $\Pp_X|_{\Jac_X \times \{ \Oo_X \}} \cong \Oo_{\Jac_X}$, so $\Ww$ is trivial. One needs to show that $\psi = \Nm_\beta$. By continuity and separatedness, it suffices to show that these maps agree on the dense open subset  $\Jac_C(q) \subset \ol\Jac_C(q)$. 

We therefore fix a line bundle $L \in \Jac_C(q)$ and evaluate both maps. By Lemma \ref{lm decomposition of olJac_C} we have a decomposition $L \cong L_X \otimes M$ for $L_X \in \Jac_X$ and $M \in \Prym_{\beta}(q)$, and the Norm map is given by $\Nm_\beta(L) = L_X^n$, which via the self-duality isomorphism $\Jac_X \cong \hat \Jac_X \cong \Jac_X/\Jac_X[n]$ is identified with $\Pp_{X,L_X}$. It then suffices to show
\begin{equation} \label{Norm = psi}
\Pp_{X,L_X} \cong i^*\Pp|_L.
\end{equation}
The right hand side is given by
\[
i^*\Pp|_L \cong i^* \Dd_{f_2} (\Uu_q \otimes f_1^*L)^{-1} \otimes i^{*} \Dd_{f_2}(f_1^*L) \otimes i^{*} \Dd_{f_2}(\Uu_q), 
\]
where $f_i$ denotes the corresponding projection appearing in \eqref{eq: f projections} and $\Uu_q$ is the universal family for the fine moduli space $\ol \Jac_C(q)$. Using the Cartesian diagram of projections and embeddings,
\begin{equation*}
\begin{tikzcd}[column sep = huge]
 & C \times \beta^{*} \Jac_X \arrow[dl, " g_1 "'] \arrow[r, " 1 \times i "] \arrow[d, " g_2 "'] & C \times \ol \Jac_C(q)  \arrow[d, " f_2 "] \\
C & \beta^{*} \Jac_X \arrow[r, " i "'] & \ol \Jac_C(q) 
\end{tikzcd},
\end{equation*}
where $g_1 := f_1 \circ (1 \times i)$ is the projection to $C$, one applies base change for determinant in cohomology \cite{esteves} to get
\[ 
i^*\Pp|_L \cong D_{g_2}( (1 \times i)^{*}\Uu_q \otimes g_1^{*}L)^{-1} \otimes D_{g_2}(g_1^{*}L) \otimes D_{g_2}( (1 \times i )^{*} \Uu_q ).
\]
By the universal property, the universal bundle $\Vv$ for $\Jac_X$ with a normalisation adapted to $\Uu$ satisfies 
\[
(1 \times i)^{*} \Uu \cong (\beta \times 1)^{*} \Vv.
\]
Substituting this alongside $L \cong L_X \otimes M$ and applying \cite[Cory. 4.4]{esteves} gives 
\[
i^*\Pp|_L \cong  \Tt_1 \otimes \Tt_2,
\]
with
\[
\Tt_1 \cong 
D_{g_2}( (\beta \times 1)^{*} \Vv \otimes g_1^{*} \beta^{*} L_X )^{-1}
\otimes D_{g_2}( g_1^{*} \beta^{*} L_X)
\otimes D_{g_2}( (\beta \times 1)^{*} \Vv),
\]
\[
\Tt_2 \cong 
D_{g_2}( (\beta \times 1)^{*} \Vv \otimes g_1^{*}M)^{-1}
\otimes D_{g_2}( g_1^{*}M )
\otimes D_{g_2}( (\beta \times 1)^{*} \Vv).
\]
Another application of base change gives $\Tt_1 \cong \Pp_{X,L_X}$ and one can also use $\Nm(M) \cong \Oo_X$ to show that $\Tt_2 \cong \Oo_{\Jac_X}$, from which (\ref{Norm = psi}) follows, concluding the proof. \end{proof}

%Since the lines bundles obtained by pull-back under $\beta$ trivialize over $\Nm_\beta^{-1}(\Oo_X) \cong \ol\Prym_\beta(q)$, one derives the following.

%\begin{corollary} \label{co Pp_gamma trivializes over Prym}
%Take a polarization $q'$ such that $\Oo_C$ is $q'$-stable, so that $\beta^* \Jac_X$ and $\ol \Prym_\beta(q')$ intersect at $\Gamma$. For every $L_\gamma \in \Gamma$, we have 
%\[
%\jmath^*\Pp_{L_\gamma} \cong \Oo_{\ol\Prym_\beta}.
%\]  
%\end{corollary}
We are now in a position to prove that $\Rr$ is $\Gamma$-equivariant.
\begin{proposition} \label{pr Gamma equivariant}
Let $C$ be a curve as in \eqref{dagger}, $X$ a smooth curve, $q$ and $q'$ general polarizations on $C$ and $\beta : C \to X$ a finite flat morphism with non-empty ramification. The sheaf $\Rr$, defined over $\ol{\Prym}_\beta(q) \times \ol{\Prym}_\beta(q')$, is flat over each factor and $\Gamma$-equivariant, with $\gamma \in \Gamma$ acting by pull-back under $\id_{\ol{\Prym}} \times \tau_\gamma$. 
\end{proposition}

We need to recall the {\it see-saw} principle before addressing the proof of Proposition \ref{pr Gamma equivariant}. Here, we reproduce the statement as in \cite[Lma. 5.5]{melo2}, adapting the hypothesis to our case.

\begin{lemma}[See-saw principle] \label{lm see-saw}
Let $Z$ and $T$ be two reduced locally Noetherian schemes with $Z$ proper and connected. Let $\Ee$ and $\Ff$ be two sheaves on $Z \times T$, flat over $T$, such that 
\begin{enumerate}
\item[(i)] $\Ff |_{Z \times \{ t \}} \cong \Ee |_{Z \times \{ t \}}$, for all $t \in T$;
\item[(ii)] $\Ff|_{Z \times \{ t \}}$ is simple for every $t \in T$;  
\item[(iii)] there exists $z_0 \in Z$ and an isomorphism of line bundles $\Ff |_{\{ z_0 \} \times T} \cong \Ee |_{\{ z_0 \} \times T}$.
\end{enumerate}
Then $\Ee$ and $\Ff$ are isomorphic. 
\end{lemma}

\begin{proof}[Proof of Proposition \ref{pr Gamma equivariant}]
The equivariance of the duality isomorphism, 
\[
\Theta_C(\beta^*L_\gamma \otimes \Ff) \cong \Theta_C^0(\beta^*L_\gamma) \otimes \Theta_C(\Ff) \cong \Nm_\beta^* \Theta_X(L_\gamma) \otimes \Theta_C(\Ff), 
\]
and since $\Nm_\beta^* \Theta_X(L_\gamma)$ trivializes over $\ol\Prym_\beta(q)$, we conclude that $\Rr$ is pointwise equivariant, {\it i.e.}
\begin{equation} \label{eq pointwise Gamma equivariance}
\Rr|_{\ol{\Prym} \times \{ \Ff \}} \cong (\id_{\ol{\Prym}} \times \tau_\gamma)^* \Rr|_{\ol{\Prym} \times \{ \Ff \}},
\end{equation}
for any $\Ff \in \ol\Prym_\beta(q)$. In order to obtain an isomorphism between $\Rr$ and $(\id_{\ol \Prym} \times \tau_\gamma)^*\Rr$, we shall now apply the see-saw principle, previously reproduced as Lemma \ref{lm see-saw}. Let us now check if the hypothesis of this lemma hold.

Thanks to Corollary \ref{co description of fine comp Pryms} $\ol\Prym_\beta(q)$ is reduced, projective (hence proper) and locally Noetherian scheme. Furthermore, $\Rr$ is flat over $\ol\Prym_\beta(q')$ thanks to Lemma \ref{lm flatness of Rr}, and since $\id_{\ol{\Prym}} \times \tau_\gamma$ is an isomorphism, $(\id_{\ol{\Prym}} \times \tau_\gamma)^* \Rr$ is also flat over the second factor, which is one of the hypothesis of the see-saw principle (Lemma \ref{lm see-saw}). Also, after \eqref{eq pointwise Gamma equivariance} one has that hypothesis (i) of Lemma \ref{lm see-saw} is satisfied. 

To check hypothesis (ii), take $\Ff \in \ol \Prym_\beta(q')$ and consider the morphism $\iota_{\Ff} : \ol  \Prym_\beta(q) \times \{ \Ff \} \hookrightarrow \ol  \Prym_\beta(q) \times \ol  \Prym_\beta(q')$, it follows from Lemma \ref{lm End Rr} that
\[
\Hhom \left (\Rr_{\Ff}, \Rr_{\Ff} \right ) \cong \Hhom \left ( \iota^*_{\Ff} \Rr, \iota^*_{\Ff} \Rr \right ) \cong \iota^*_{\Ff}\Hhom(\Rr, \Rr) \cong \Oo_{\ol\Prym \times\{ \Ff \}},
\]
implying that hypothesis (ii) of Lemma \ref{lm see-saw} is satisfied.

We recall that $\Pp$ is normalized at $\{ \Oo_C \}$, so $\Rr|_{\{ \Oo_C \} \times \ol\Prym} \cong \Oo_{\ol\Prym}$, hence 
\[
\Rr|_{\{ \Oo_C \} \times \ol\Prym} \cong \Oo_{\ol\Prym} \cong \tau_\gamma^* \Oo_{\ol\Prym} \cong \tau_\gamma^* \Rr|_{\{ \Oo_C \} \times \ol\Prym},
\]
what implies that (iii) holds as well.

Since all of the hypothesis of Lemma \ref{lm see-saw} hold, one concludes that
\[
\Rr \cong \left (\id_{\ol\Prym} \times \tau_\gamma \right )^* \Rr
\]
for all $\gamma \in \Gamma$. Choosing a set of generators of $\Gamma$ one can equip $\Rr$ with a $\Gamma$-equivariant structure.
\end{proof}

\subsection{An equivalence of categories} \label{equivalence equivariant derived category}
With the sheaf $\Rr := (\jmath \times \jmath')^{*} \Pp$ on $\ol \Prym_\beta(q) \times \ol \Prym_\beta(q')$ as an integral kernel and the projections onto the first and second factors
\begin{equation} \label{eq r projections}
\xymatrix{
	&  \ol{\Prym}_\beta(q) \times \ol{\Prym}_\beta(q') \ar[ld]_{r_1} \ar[rd]^{r_2} &
	\\
	\ol{\Prym}_\beta(q) & & \ol{\Prym}_\beta(q'),}
\end{equation}
we may define the integral functor 
\begin{equation} \label{eq Phi^R}
\morph{D^b \left ( \ol{\Prym}_\beta(q) \right )}{D^b \left ( \ol{\Prym}_\beta(q') \right )}{\Ff^\bullet}{R r_{2,*}(r_1^*\Ff^\bullet \otimes \Rr).}{}{\Phi^\Rr_{\ol\Prym \to \ol\Prym'}}
\end{equation}
Now consider the equivariant dervied category $D^b(\ol \Prym_\beta(q'), \Gamma):= D^b(\Coh(\ol \Prym_\beta(q'))^{ \Gamma})$, which is the derived category of $\Gamma$-equivariant sheaves on $\ol \Prym_\beta(q')$. The usual derived functors may be defined and satisfy adjunction, the projection formula, flat base change and Grothendieck-Verdier duality. For more details see for instance \cite[Section 4]{bridgeland&king&reid}. 

Considering the trivial $\Gamma$-action on $\ol{\Prym}_\beta(q)$, one can identify $D^b( \ol{\Prym}_\beta(q), \Gamma) = D^b( \ol{\Prym}_\beta(q))$ by means of the forgetful functor $\zeta_0:D^b( \ol{\Prym}_\beta(q), \Gamma ) \to D^b( \ol{\Prym}_\beta(q) )$. With the tensoral $\Gamma$-action on $\ol \Prym_\beta(q')$, $\Rr$ is $\Gamma$-equivariant with respect to the induced action on $\ol \Prym_\beta(q) \times \ol \Prym_\beta(q')$, as shown in Proposition \ref{pr Gamma equivariant}, so $\Rr$ can be equipped with a $\Gamma$-equivariant structure and has a natural interpretation as an element of $D^b(\ol \Prym_\beta(q) \times \ol \Prym_\beta(q'), \Gamma)$. We may therefore define the integral functor on the equivariant derived categories, 
\begin{equation} \label{eq Psi^R}
\morph{D^b \left ( \ol{\Prym}_\beta(q) \right ) = D^b \left ( \ol{\Prym}_\beta(q) , \Gamma \right )}{D^b \left ( \ol{\Prym}_\beta(q'), \Gamma \right )}{\Ff^\bullet}{R r_{2,*}(r_1^*\Ff^\bullet \otimes \Rr),}{}{\Psi^\Rr_{\ol\Prym \to \ol\Prym'}}
\end{equation}
which lifts to \eqref{eq Phi^R} via the forgetful functor $\zeta:D^b( \ol{\Prym}_\beta(q'), \Gamma ) \to D^b( \ol{\Prym}_\beta(q') )$, in the sense that we get a commutative square 
\begin{equation*}
\begin{tikzcd}[column sep = huge]
D^b( \ol{\Prym}_\beta(q) ) \arrow[r, "\Phi^\Rr_{\ol\Prym \to \ol\Prym'}"] &  D^b( \ol{\Prym}_\beta(q') ) \\
D^b( \ol{\Prym}_\beta(q), \Gamma ) \arrow[r, "\Psi^\Rr_{\ol\Prym \to \ol\Prym'}"'] \ar["\zeta_0", "\simeq"', u] & D^b( \ol{\Prym}_\beta(q'), \Gamma ) \arrow[u, "\zeta"']
\end{tikzcd}.
\end{equation*}

The main theorem of this article is that $\Psi^{\Rr}_{\ol \Prym \to \ol \Prym'}$ is a Fourier-Mukai transform. 

\begin{theorem} \label{tm fourier-mukai general case}
Suppose that $C$ is a curve as in \eqref{dagger} with a finite flat morphism $\beta:C \to X$ with non-empty ramification. The integral functor 
\[
\Psi^{\Rr}_{\ol \Prym \to \ol \Prym'}: D^b(\ol \Prym_\beta(q)) \to D^b(\ol \Prym_\beta(q'),{\Gamma}),
\]
as defined in \eqref{eq Psi^R}, is an equivalence of categories. 
\end{theorem}

The remainder of this section is dedicated to the proof. Our main tool is the following criterion of Bridgeland. 

\begin{theorem}[\cite{bridgeland}, Th. 3.3]
\label{tm bridgeland equivalence}
Let $\Aa$ and $\Bb$ be triangulated categories, and let $F: \Aa \to \Bb$ be a fully faithful exact functor. Suppose that $\Bb$ is indecomposable, and that not every object of $\Aa$ is isomorphic to $0$. Then $F$ is an equivalence of categories if and only if $F$ has a left adjoint $G$ and a right adjoint $H$ such that for any object $B$ of $\Bb$, $H(B)\cong 0$ implies $G(B)\cong 0$.
\end{theorem}

The existence of left and right adjoints to $\Psi^{\Rr}_{\ol \Prym \to \ol \Prym'}$ follows from an equivariant modification to work of Hern\'andez-Ruip\'erez--Lopez-Mart\'{\i}n--Sancho-de-Salas \cite{HLS}, which has been widely generalised by Rizzardo \cite{rizzardo}.  

\begin{proposition}[\cite{HLS}, Prop. 1.17] \label{pr adjoints}
Let $X$ and $Y$ be Gorenstein, projective schemes of finite type with dualising sheaves $\omega_X$, $\omega_Y$ and dimensions $m=\dim{X}$, $n=\dim{Y}$. If $\Kk$ is a complex of finite projective dimension on $X\times Y$, flat over each factor, then the integral functor $\Phi^{\Kk}_{X\rightarrow Y} : D^b(X) \rightarrow D^b(Y)$ has the following left and right adjoints: 
\begin{equation*}
    \adjointpair{\adjointpair{\Phi^{\Kk^\vee \otimes \pi_Y^\ast \omega_Y[n]}_{Y\rightarrow X}}{\Phi^\Kk_{X\rightarrow Y}}}{\Phi^{\Kk^\vee \otimes \pi_X^\ast \omega_X[m]}_{Y\rightarrow X}}.
\end{equation*}

\end{proposition}
We first apply this to the integral functor $\Phi^\Rr_{\ol\Prym \to \ol\Prym'}$ from \eqref{eq Phi^R}. With the notation 
\[
\morph{D^b \left ( \ol{\Prym}_\beta(q') \right )}{D^b \left ( \ol{\Prym}_\beta(q) \right )}{\Ff^\bullet}{R r_{1,*}(r_2^*\Ff^\bullet \otimes \Rr^\vee[g]).}{}{H := \Phi^{\Rr^\vee[g]}_{\ol \Prym' \to \ol\Prym}},
\]
it follows from Proposition \ref{pr adjoints} and triviality of the dualising sheaves (Corollary \ref{co description of fine comp Pryms}, Proposition \ref{pr quotient prym}) that the integral funtor from \eqref{eq Phi^R} has the following left and right adjoints: 
\begin{equation*} %\label{eq same adjoints for H}
\adjointpair{\adjointpair{H}{\Phi^\Rr_{\ol\Prym \to \ol\Prym'}}}{H}.
\end{equation*}
We take a $\Gamma$-equivariant modification,
\[
H^\Gamma := \zeta_0^{-1} \circ \Phi_{\ol \Prym' \to \ol \Prym}^{\Rr^\vee[g]} \circ \zeta : D^b(\ol \Prym_\beta(q'), \Gamma) \to D^b(\ol \Prym_\beta(q), \Gamma) = D^b \left ( \ol{\Prym}_\beta(q)\right ).
\]
\begin{lemma} \label{le adjoints}
The functor $\Psi^{\Rr}_{\ol \Prym \to \ol \Prym'}$ defined in \eqref{eq Psi^R} has the adjoints 
\begin{equation*} %\label{eq same adjoints for H^Gamma}
\adjointpair{\adjointpair{H^{\Gamma}}{\Psi^\Rr_{\ol\Prym \to \ol\Prym'}}}{H^{\Gamma}}.
\end{equation*}
In particular, the left and right adjoints to $\Psi^{\Rr}_{\ol \Prym \to \ol \Prym'}$ coincide.
\end{lemma}
\begin{proof}
We make a $\Gamma$-equivariant modification to the proof of \cite[Propn. 1.17]{HLS}, which provides a detailed reference for the forthcoming calculation. Take $\Ff^\bullet \in D^b(\ol \Prym_\beta(q))$ and $\Gg^\bullet \in D^b(\ol \Prym_\beta(q'), \Gamma)$. Note that pullback and pushforward along the projections in \eqref{eq r projections} preserves derived categories of $\Gamma$-equivariant complexes. The equivariant derived category has the push-pull adjunction 
\[ 
\Hom_{D^b(\ol \Prym_\beta(q'), \Gamma)}(\Gg^\bullet, \Psi^{\Rr}_{\ol \Prym \to \ol \Prym'}\Ff^\bullet) 
\cong \Hom_{D^b(\ol \Prym_\beta(q) \times \ol \Prym_\beta(q'), \Gamma)}(r_2^{*}\Gg^\bullet, r_1^{*}\Ff^\bullet \otimes \Rr).
\]
The dualising functor $(\_)^{\vee}= \Hh om(\Oo, \_ ) $ acts on the equivariant derived category and satisfies the usual compatibilities. We therefore get  
\begin{align*}
\Hom_{D^b(\ol \Prym_\beta(q'), \Gamma)}(\Gg^\bullet, \Psi^{\Rr}_{\ol \Prym \to \ol \Prym'}\Ff^\bullet) 
& \cong 
\Hom_{D^b(\ol \Prym_\beta(q) \times \ol \Prym_\beta(q'), \Gamma)}((r_1^{*}\Ff^\bullet \otimes \Rr)^{\vee}, (r_2^{*}\Gg^\bullet)^{\vee}), \\
& \cong 
\Hom_{D^b(\ol \Prym_\beta(q) \times \ol \Prym_\beta(q'), \Gamma)}( r_1^{*}(\Ff^\bullet)^{\vee} , (\Rr^{\vee} \otimes (r_2^{*}\Gg^\bullet))^{\vee} )).
\end{align*}
Applying push-pull adjunction followed by Grothendieck--Verdier duality with the trivial dualising sheaf, we get 
\begin{align*}
\Hom_{D^b(\ol \Prym_\beta(q'), \Gamma)}(\Gg^{\bullet}, \Psi^{\Rr}_{\ol \Prym \to \ol \Prym'}\Ff^{\bullet})
& \cong 
\Hom_{D^b(\ol \Prym_\beta(q))}( (\Ff^\bullet)^{\vee}, Rr_{1,*}(\Rr^{\vee} \otimes r_2^{*} \Gg^\bullet)^{\vee} ), \\
& \cong 
\Hom_{D^b(\ol \Prym_\beta(q))}( Rr_{1,*} (\Rr^{\vee}[g] \otimes r_2^{*}\Gg^\bullet), \Ff^\bullet), 
\\
& = \Hom_{D^b(\ol \Prym_\beta(q))}( H^{\Gamma}(\Gg^\bullet), \Ff^\bullet)
\end{align*}
The adjoint on the other side is proven similarly and this completes the lemma.
\end{proof}

To prove the fully faithfulness of $\Psi^{\Rr}_{\ol \Prym \to \ol \Prym'}$ we shall apply a criterion of Bridgeland in terms of spanning classes.

\begin{theorem}[\cite{bridgeland}, Th. 2.3] \label{tm bridgeland full faithfullness}
Let $\Aa$ and $\Bb$ be triangulated categories, and let $F: \Aa \to \Bb$ be an
exact functor with a left and a right adjoint. Then $F$ is fully faithful if and only if there
exists a spanning class $\Omega$ for $\Aa$ such that for all elements $A_1, A_2 \in \Omega$ and all integers $i$, the map
\begin{equation} \label{hom test}
F: \Hom(A_1, A_2[i]) \longrightarrow \Hom(F(A_1),F(A_2)[i])
\end{equation}
is an isomorphism.
\end{theorem}

Recall that $\jmath: \ol \Prym_\beta(q) \hookrightarrow \ol \Jac_C(q)$ and $\jmath': \ol \Prym_\beta(q') \hookrightarrow \ol \Jac_C(q')$ denote the closed embeddings. Our strategy to show that $\Psi^\Rr_{\ol \Prym \to \ol \Prym'}$ satisfies the hypothesis of Theorem \ref{tm bridgeland full faithfullness} consists of exploiting the following relation. 

\begin{lemma} \label{le relation between integral functors} Consider the integral functors $\Phi^{\Pp}_{\ol \Jac \to \ol \Jac'}$ and $\Phi^{\Rr}_{\ol \Prym \to \ol \Prym'}$ constructed in \eqref{eq FM Arinkin} and \eqref{eq Phi^R} respectively. Then, there is a natural equivalence 
\[
\Phi^{\Rr}_{\ol \Prym \to \ol \Prym'} \simeq L \jmath'^{*} \circ \Phi^{\Pp}_{\ol \Jac \to \ol \Jac'} \circ R\jmath_{*}
\]
\end{lemma}
\begin{proof}
We record a composition of natural equivalences from base change, the projection formula and functoriality. With the Cartesian diagram
\begin{equation*}
\begin{tikzcd}[column sep = huge]
\ol \Prym_\beta(q) \times \ol \Jac_C(q') \arrow[r, "\jmath \times 1"] \arrow[d, "t_1"'] & \ol \Jac_C(q) \times \ol \Jac_C(q') \arrow[d, "p_1"] \\
\ol \Prym_\beta(q) \arrow[r, "\jmath"'] & \ol \Jac_C(q)
\end{tikzcd},
\end{equation*}
proper base change provides the isomorphism 
\[
L\jmath'^{*} \circ \Phi^{\Pp}_{\ol \Jac_C \to \ol \Jac_C'} \circ R\jmath_{*} (\Ff^{\bullet}) 
= L\jmath'^{*} \circ Rp_{2,*}( (p_1^{*} \circ R\jmath_{*}\Ff^{\bullet}) \otimes \Pp ) 
\cong L \jmath'^{*} \circ Rp_{2,*} ( (R (\jmath \times 1)_{*} \circ t_1^{*} \Ff^{\bullet}) \otimes \Pp ).
\]
Apply the projection formula to get
\[
L \jmath'^{*} \circ \Phi^{\Pp}_{\ol \Jac_C \to \ol \Jac_C'} \circ R\jmath_{*} (\Ff^{\bullet}) 
\cong L \jmath'^{*} \circ R(p_2 \circ (\jmath \times 1))_{*}( t_1^{*}\Ff^{\bullet} \otimes L(\jmath \times 1)^{*} \Pp ),
\]
followed by proper base change around the Cartesian diagram
\begin{equation*}
\begin{tikzcd}[column sep = huge]
\ol \Prym_\beta(q) \times \ol \Prym_\beta(q') \arrow[r, "1 \times \jmath'"] \arrow[d, "r_2"'] & \ol \Prym_\beta(q) \times \ol \Jac_C(q') \arrow[d, "p_2 \circ (\jmath \times 1)"] \\
\ol \Prym_\beta(q') \arrow[r, "\jmath'"'] & \ol \Jac_C(q')
\end{tikzcd},
\end{equation*}
which provides us with
\[
L \jmath'^{*} \circ \Phi^{\Pp}_{\ol \Jac_C \to \ol \Jac_C'} \circ R\jmath_{*} (\Ff^{\bullet}) 
\cong Rr_{2,*} \circ L(1 \times \jmath')^{*}( t_1^{*}\Ff^{\bullet} \otimes L(\jmath \times 1)^{*} \Pp ).
\]
Bringing pullback inside the tensor product and recalling that by Lemma \ref{derived pullback is pullback} we have $\Rr \cong L(\jmath \times \jmath')^{*}\Pp$, one obtains 
\[
L \jmath'^{*} \circ \Phi^{\Pp}_{\ol \Jac_C \to \ol \Jac_C'} \circ R\jmath_{*} (\Ff^{\bullet}) 
\cong Rr_{2,*}( L (1\times \jmath')^{*} \circ t_1^{*} \Ff^{\bullet} \otimes L (1\times \jmath')^{*} \circ L(\jmath \times 1)^{*} \Pp )
\cong Rr_{2,*} ( r_1^{*} \Ff^{\bullet} \otimes \Rr)
\]
which concludes the proof. 
\end{proof}

\begin{lemma}
\label{le Hom sets and points}
For all closed points $x, y \in \ol \Prym_\beta(q)$, and all $i \in \ZZ$, there exists an isomorphism 
\[
\Hom_{D^b(\ol \Prym_\beta(q'), \Gamma)}(\Psi^{\Rr}_{\ol \Prym \to \ol \Prym'}(\Oo_x), \Psi^{\Rr}_{\ol \Prym \to \ol \Prym'}(\Oo_{y})[i]) \cong \Hom_{D^b(\ol \Jac_C(q'))}(\Phi^{\Pp}_{\ol \Jac \to \ol \Jac'} \circ R\jmath_{*}(\Oo_x) , \Phi^{\Pp}_{\ol \Jac \to \ol \Jac'} \circ R\jmath_{*}(\Oo_{y})[i])
\]
\end{lemma}

\begin{proof}
If $x \neq y$ or $i \neq 0$, then one only has the zero morphism on both sides. If $x = y$ and $i=0$, our Hom-space calculations rely on the ambient spaces $\ol{\Prym}_\beta(p) \subset \ol{\Jac}_C(p)$ for both $p = q$ and $p = q'$. Since $\Phi^\Pp_{\ol\Jac \to \ol\Jac'}$ is an equivalence, note that 
\[
\Hom_{D^b(\ol \Jac_C(q'))}(\Phi^\Pp_{\ol\Jac \to \ol\Jac'} \circ R\jmath_{*} (\Oo_x) , \Phi^\Pp_{\ol\Jac \to \ol\Jac'} \circ R\jmath_{*}(\Oo_{x})) = k,
\]
which implies  
\[
\Hom_{D^b(\ol\Prym_\beta(q'))}( L\jmath^{*} \circ \Phi^\Pp_{\ol\Jac \to \ol\Jac'} \circ R\jmath_{*}(\Oo_x), L\jmath^{*} \circ \Phi^\Pp_{\ol\Jac \to \ol\Jac'} \circ R\jmath_{*}(\Oo_{x})) = k,
\]
as any non-scalar $L\jmath^* \eta$ would give $\eta$ non-scalar. From Lemma \ref{le relation between integral functors} it follows that 
\[
\Hom_{D^b(\ol\Prym_\beta(q'))}( \Phi^{\Rr}_{\ol \Prym \to \ol \Prym'}(\Oo_x) , \Phi^{\Rr}_{\ol \Prym \to \ol \Prym'}(\Oo_x) ) 
= k.
\]
Since all rescalings are equivariant, this restricts to the $\Gamma$-equivariant dervied category as  
\[
\Hom_{D^b(\ol\Prym_\beta(q'), \Gamma)}( \Psi^{\Rr}_{\ol \Prym \to \ol \Prym'}(\Oo_x) , \Psi^{\Rr}_{\ol \Prym \to \ol \Prym'}(\Oo_x) ) = k,
\]
which completes the proof. 
\end{proof}

We now collect and apply the tools studied in this section to prove our main theorem.

\begin{proof} 
(\textit{of Theorem \ref{tm fourier-mukai general case}}).
We verify the hypothesis of Theorem \ref{tm bridgeland equivalence}. The property involving adjoints is provided by Lemma \ref{le adjoints}, and the target category $D^b \left ( \ol{\Prym}_\beta(q'), \Gamma \right )$ is indecomposible by \cite[Lma. 4.2]{bridgeland&king&reid}. It remains to check that $\Psi^\Rr_{\ol \Prym \to \ol \Prym'}$ is fully faithful. Applying Theorem \ref{tm bridgeland full faithfullness}, we shall check that $\Psi^\Rr_{\ol\Prym \to \ol\Prym'}$ respects the homomorphisms between elements of a given spanning class in the derived category of $\ol\Prym_\beta(q)$. Since $\ol\Prym_\beta(q)$ is Gorenstein by Corollary \ref{co description of fine comp Pryms}, the skyscrapers $\Oo_x$ for all closed points $x \in \ol\Prym_\beta(q)$ form a spanning class \cite[Lma. 1.26]{HLS}. 
% The push-forward $R\jmath_*\Oo_x$ are naturally skyscrapers in $\ol\Jac_C(q)$, so they belong to a spanning class there also by \cite[Lma. 1.26]{HLS} and Theorem \ref{tm fine comp Jacobians}. 
This reduces the proof to the computation  
\begin{equation*}
\Hom_{D^b(\ol \Prym_\beta(q'), \Gamma)}(\Psi^{\Rr}_{\ol \Prym \to \ol \Prym'}(\Oo_x), \Psi^{\Rr}_{\ol \Prym \to \ol \Prym'}(\Oo_{y})[i]) 
\cong \Hom_{D^b(\ol \Prym_\beta(q))}(\Oo_x ,  \Oo_{y}[i])
\end{equation*}
for every pair of closed points $x, y \in \ol\Prym_\beta(q)$ and every $i \in \ZZ$. This follows from the composition 
\begin{align*}
\Hom_{D^b(\ol \Prym_\beta(q'), \Gamma)}(\Psi^{\Rr}_{\ol \Prym \to \ol \Prym'}(\Oo_x), & \Psi^{\Rr}_{\ol \Prym \to \ol \Prym'}(\Oo_{y})[i]) \\
& \cong \Hom_{D^b(\ol \Jac_C(q'))}(\Phi^{\Pp}_{\ol \Jac \to \ol \Jac'} \circ R\jmath_{*}(\Oo_x) , \Phi^{\Pp}_{\ol \Jac \to \ol \Jac'} \circ R\jmath_{*}(\Oo_{y})[i]), \\
& \cong \Hom_{D^b(\ol \Jac_C(q))}(R\jmath_{*}(\Oo_x) ,  R\jmath_{*}(\Oo_{y})[i]), \\
& \cong \Hom_{D^b(\ol \Prym_\beta(q))}(\Oo_x ,  \Oo_{y}[i]).
\end{align*}
The first isomorphism is Lemma \ref{le Hom sets and points}. The second isomorphism is due to the fact that $\Phi^{\Pp}_{\ol \Jac \to \ol \Jac'}$ defines a derived equivalence (Theorem \ref{thm:derivedequiv}), and so must satisfy the criterion of Theorem \ref{tm bridgeland full faithfullness}. The third isomorphism is push-pull adjunction. 
\end{proof}

%\begin{corollary}
%Over irreducible curves of type (\ref{dagger}) the Fourier-Mukai transform $\Phi^{\Qq}_{\ol \Prym \to \ol \Prym / \Gamma}$ induces an isomorphism 
%\begin{equation} \label{eq duality of SL PGL fibres}
%\ol \Prym_\beta^{\vee} \cong \ol \Prym_\beta / \Gamma.
%\end{equation}
%\end{corollary}

%\begin{proof}
%We observe a universal property on $\Qq$ as a family of sheaves and use this to induce a morphism on the moduli spaces at hand. Pick a closed point $[x] \in \ol \Prym_\beta / \Gamma$ and a lift $x \in \ol \Prym_\beta$. 
%\[
%\Qq |_{ \ol \Prym_\beta \times [x] } 
%\cong (1 \times \pi)^{*} \Qq | _{ \ol \Prym_\beta \times x }
%\cong (\jmath \times \jmath)^{*}\Pp |_{\ol \Prym_{\beta} \times [x] }
%\cong \jmath^{*}\Pp |_{\ol \Jac_C \times x} 
%\]
%By \cite[Theorem B]{arinkin} this slice is a stable rank one torsion-free sheaf on $\ol \Prym_\beta$ in the connected component of $\Oo_{\ol \Prym_\beta}$, thus giving rise to a morphism 
%\[
%\ol \Prym_\beta / \Gamma \to \ol \Prym_\beta^{\vee},  \quad [x] \mapsto \Qq |_{ \ol \Prym_\beta \times [x] } 
%\]
%which is just the restriction of Arinkin's morphism. The inverse transform provides the inverse morphism via a similar construction. 
%\end{proof}
%The difficulty in extending this corollary to curves which are reducible arises from $\ol \Prym_\beta(q)$ no longer being irreducible. The several components make (\ref{eq duality of SL PGL fibres}) sensitive to a choice of polarisation on $\ol \Prym_\beta(q)$, as this changes the moduli problem of $\ol \Prym_\beta(q)^{\vee}$. 

\section{The role of endoscopy} 
\label{sec FM and compactified Prym}
\label{FM to quotient}
\label{descent and endoscopy}

For an $n:1$ finite flat morphism $\beta : C \to X$, the non-endoscopy condition is defined by $\Gamma = \Jac_X[n]$ acting freely on $\ol\Prym_\beta(q')$. In this case all $\Gamma$-equivariant complexes descend to $\ol\Prym_\beta(q')/\Gamma$, so the essential image $\Ii$ of the flat morphism $\pi$,
\[
\pi^* : D^b(\ol\Prym_\beta(q')/\Gamma) \longrightarrow D^b (\ol\Prym_\beta(q'))
\] 
satisfies $\Ii \cong D^b (\ol\Prym_\beta(q'), \Gamma)$, thus providing an equivalence of categories 
\[
D^b(\ol \Prym_\beta(q') / \Gamma) \cong D^b(\ol \Prym_\beta(q'),\Gamma),
\]
where the inverse is $L\pi_*^{\Gamma}$, the $\Gamma$-equivariant pushforward functor,
as in \cite[Thm. 1.3, Cory. 2.6]{nevins}. This section provides a Fourier-Mukai interpretation to the composition of Theorem \ref{tm fourier-mukai general case} with descent to the quotient: 
\begin{equation*} 
\begin{tikzcd}[column sep = huge]
D^b(\ol \Prym_\beta(q)) \arrow[r, "\Phi^{\Rr}_{\ol \Prym \to \ol \Prym'}"]  
& D^b(\ol \Prym_\beta(q'),\Gamma) \arrow[r, "L\pi_*^{\Gamma}"] 
& D^b(\ol  \Prym_\beta(q') / \Gamma) 
\end{tikzcd}.
\end{equation*}

We begin by specifying the integral kernel. Recall $\Rr$ from the previous section, defined by the pullback $\Rr := (\jmath \times \jmath')^{*}\Pp$ of the Poincar\'e sheaf $\Pp$ along the embeddings $\jmath : \ol{\Prym}_\beta(q) \hookrightarrow \ol{\Jac}_C(q)$ and $\jmath' : \ol{\Prym}_\beta(q') \hookrightarrow \ol{\Jac}_C(q')$. Since $\Rr$ is $\Gamma$-equivariant (Proposition \ref{pr Gamma equivariant}), it follows that $\Rr$ descends to a sheaf $\Qq$ on $\ol \Prym_{\beta}(q) \times \ol \Prym_{\beta}(q') / \Gamma$, for instance by a (vacuous) application of a descent criterion of Nevins \cite[Thm. 1.2]{nevins}. By the flatness of the quotient map $\pi: \Prym_\beta(q') \to \Prym_\beta(q') / \Gamma$, it follows that $\Qq$ is flat over both factors.
 
\begin{corollary} \label{co descent for non-endoscopic curves}
Let $C$ be a curve as in \eqref{dagger} equipped with a finite flat $n:1$ $\beta : C \to X$ of non-endoscopic type. There exists a sheaf $\Qq$ over $\ol{\Prym}_\beta(q) \times \ol{\Prym}_\beta(q')/\Gamma$, flat over each factor, satisfying 
\[
\Rr \cong ( \id_{\ol{\Prym}} \times \pi)^{*} \Qq,
\]
which is to say that $\Rr$ descends along the quotient map $1\times \pi : \ol{\Prym}_\beta(q) \times \ol{\Prym}_\beta(q') \to \ol{\Prym}_\beta(q) \times \ol{\Prym}_\beta(q')/\Gamma$, and therefore 
\begin{equation} \label{eq relation of Poincare sheaves non-endoscopic}
(\jmath \times \id_{\ol{\Jac}})^* \Pp \cong (\id_{\ol{\Prym}} \times \chi)^* \Qq
\end{equation}
holds.\end{corollary}
With the projections onto the first and second factors
\begin{equation}
\xymatrix{
	&  \ol{\Prym}_\beta(q) \times \ol{\Prym}_\beta(q')/\Gamma  \ar[ld]_{\q_1} \ar[rd]^{\q_2} &
	\\
	\ol{\Prym}_\beta(q) & & \ol{\Prym}_\beta(q')/\Gamma,
}
\end{equation}
we define the integral functor with kernel $\Qq$, 
\begin{equation} \label{eq Phi^Q}
\morph{D^b \left ( \ol{\Prym}_\beta(q) \right )}{D^b \left ( \ol{\Prym}_\beta(q')/\Gamma \right )}{\Ff^\bullet}{R \q_{2,*}(\q_1^*\Ff^\bullet \otimes \Qq),}{}{\Phi^\Qq_{\ol\Prym \to \ol\Prym'/\Gamma}}
\end{equation}
which is well-defined as $\Qq$ is a coherent sheaf flat over $\ol{\Prym}_\beta(q)$. By an application of flat base change this fits into the commutative triangle
\begin{equation} \label{eq triangle of transforms}
\begin{tikzcd}[column sep = huge]
D^b(\ol \Prym_\beta(q)) \arrow[r, "\Phi^{\Rr}_{\ol \Prym \to \ol \Prym'}"] \arrow[dr, "\Phi^{\Qq}_{\ol \Prym \to \ol \Prym' / \Gamma}"'] & D^b(\ol \Prym_\beta(q'),\Gamma) \arrow[bend right=30,swap]{d}{\pi_*^\Gamma}
 \\
 & D^b(\ol  \Prym_\beta(q') / \Gamma) \arrow[bend right=30,swap]{u}{\pi^*}
\end{tikzcd},
\end{equation}
and it follows that $\Phi^{\Qq}_{\ol \Prym \to \ol \Prym' / \Gamma}$ is a composition of derived equivalences. The following is therefore a consequence of Theorem \ref{tm fourier-mukai general case}. 

\begin{theorem} \label{tm main theorem}
Suppose that $C$ is a curve as in \eqref{dagger} equipped with a finite flat $n:1$ morphism $\beta : C \to X$ of non-endoscopic type. Consider two general polarizations $q$ and $q'$ on $C$. Then, the integral functor 
\begin{equation*} 
 \Phi^\Qq_{\ol\Prym \to \ol\Prym'/\Gamma} : D^b \left ( \ol{\Prym}_\beta(q) \right )\to D^b \left ( \ol{\Prym}_\beta(q')/\Gamma \right ),
\end{equation*}
as defined in \eqref{eq Phi^Q}, is an equivalence of categories. 
\end{theorem}

To complete the circle of ideas in this paper we relate $\Phi^\Qq_{\ol\Prym \to \ol\Prym'/\Gamma}$ back to the Fourier--Mukai transform on compactified Jacobian varieties from Theorem \ref{thm:derivedequiv}.

\begin{lemma} \label{le key properties of functors}
Consider the Fourier--Mukai transforms $\Phi^\Pp_{\ol\Jac \to \ol\Jac'}$ and $\Phi^\Qq_{\ol\Prym \to \ol\Prym'/\Gamma}$ constructed in \eqref{eq FM Arinkin} and \eqref{eq Phi^Q}, and recall the smooth morphism $\chi : \ol\Jac_C(q') \to \ol\Prym_\beta(q')/\Gamma$ defined in \eqref{eq chi} and the closed embedding $\jmath: \ol\Prym_\beta(q) \hookrightarrow \ol\Jac_C(q)$. Then,
    \begin{equation*} %\label{eq FM relation 1}
\Phi^\Pp_{\ol\Jac \to \ol\Jac'} \circ R \jmath_* \cong L\chi^* \circ \Phi^\Qq_{\ol\Prym \to \ol\Prym'/\Gamma}.
\end{equation*}
\end{lemma}

\begin{proof}
Consider the following Cartesian diagram,
\begin{equation*}
\begin{tikzcd}[column sep = huge]
\ol{\Prym}_\beta(q) \times \overline{\Jac}_C(q') \arrow[r, "\id_{\ol{\Prym}}\times \chi"] \arrow[d, "\p_2\circ \left(\jmath\times \id_{\overline{\Jac}}\right)"'] & \ol{\Prym}_\beta(q) \times \quotient{\ol{\Prym}_\beta(q')}{\Gamma} \arrow[d, "\q_2"]\\
\overline{\Jac}_C(q') \arrow[r, "\chi"'] & \quotient{\ol{\Prym}_\beta(q')}{\Gamma}
\end{tikzcd}.
\end{equation*}

Let $\Ff^\bullet \in D^b(\ol{\Prym}_\beta(q))$. Then, performing base change, recalling the relation between $\Qq$ and $\Pp$ in \eqref{eq relation of Poincare sheaves non-endoscopic} and applying projection formula, one gets the identification 
\[
\left(L\chi^\ast \circ \Phi^\Qq_{\ol\Prym \to \ol\Prym'/\Gamma} \right)(\Ff^\bullet) \cong R\p_{2,\ast}\left(R(\jmath\times 1_{\overline{\Jac}_C})_\ast\left(L(\q_1\circ (\id_{\ol{\Prym}}\times \chi))^\ast \Ff^\bullet \right)\otimes^L \Pp\right)
\] 
Now, considering a base change for the Cartesian diagram
\begin{equation*}
\begin{tikzcd}[column sep = huge]
\ol{\Prym}_\beta(q) \times \overline{\Jac}_C(q') \arrow[r, "\q_1\circ \left(\id_{\ol{\Prym}}\times \chi\right)"] \arrow[d, "\jmath\times \id_{\overline{\Jac}_C}"'] & \ol{\Prym}_\beta(q) \arrow[d, "\jmath"]\\
\overline{\Jac}_C(q) \times \overline{\Jac}_C(q') \arrow[r, "\p_1"'] & \overline{\Jac}_C(q)
\end{tikzcd},
\end{equation*}
the expression above can be rewritten,
\begin{equation*}
	\left(L\chi^\ast \circ \Phi^\Qq_{\ol\Prym \to \ol\Prym'/\Gamma}\right)(\Ff^\bullet) \cong R\p_{2,\ast}\left(\p_1^\ast (R\jmath_\ast \Ff^\bullet)\otimes^L \Pp\right),
\end{equation*}
and the proof follows.
\end{proof}

Finally we comment on the key difference between the endoscopic and non-endoscopic cases. With endoscopic $\beta : C \to X$, $\Gamma$ does not act freely on $\ol \Prym_{\beta}(q)$, so the quotient morphism $\pi: \ol \Prym_\beta(q) \to \ol \Prym_\beta(q) / \Gamma$ is not flat. One may still construct the sheaf $\Qq$ on $\ol \Prym_\beta(q) \times \ol \Prym_\beta(q') / \Gamma$ via the descent criterion of Nevins, however $\Qq$ will fail to be flat over both factors. Proposition \ref{pr adjoints} no longer applies and adjointness of the integral functor fails. Therefore in the endoscopic case the corresponding integral functor would fail to be an equivalence of categories.

One still gets the Diagram \eqref{eq triangle of transforms} but the push-pull functor pair are no longer equivalences. Instead the subcategories $\Kk := \ker(\pi_*^{\Gamma})$ and $\Ii$ the essential image of $\pi^{*}$ form a semiorthogonal decomposition of $D^b(\ol \Prym_\beta(q'),\Gamma)$,
\[
D^b(\ol \Prym_\beta(q'),\Gamma) = \langle \Kk, \Ii \rangle,
\]
where the component $\Kk$ characterises the sheaves that do not descend to the quotient (\cite[Propn. 2.8]{nevins}). This describes the situation of {\it fractional branes}, as studied by Frenkel and Witten in \cite{frenkel&witten} over the endoscopic locus of the Hitchin fibration.

%%%%%%%%%%%%%%%%%%%%%%%%%%%%%%%%%%%%%%%%%%%%%%%%%%%%%%%%%%%%%%%%%%%%%%%%%%%%


\newpage


\begin{thebibliography}{99}

\bibitem{altman&kleiman} 
A. B. Altman and S. L. Kleiman. 
{\it Compactifying the Picard scheme}. 
Adv. Math. 35 (1980), 50--112.

\bibitem{arinkin1}
D. Arinkin. 
{\it Cohomology of line bundles on compactified Jacobians}.
Math. Res. Lett. 18 (2011), 6, 1215--1226.

\bibitem{arinkin}
D. Arinkin.
{\it Autoduality of compactified Jacobians for curves with plane singularities}. 
J. Algebraic Geometry, {\bf 22} (2013), 363--388.

\bibitem{beauville&arnaud&ramanan}
A. Beauville, M. S. Narasimhan and S. Ramanan.
{\it Spectral curves and the generalised theta divisor.}. 
J. reine angew. Math. 398,  (1989), 169--179.

\bibitem{bridgeland}
T. Bridgeland. 
{\it Equivalences of triangulated categories and Fourier--Mukai transforms}. 
Bull. London Math. Soc., 31 (1999), pp. 25--34.


\bibitem{bridgeland&king&reid}
T. Bridgeland. 
{\it The McKay correspondence as an equivalence of derived categories}. 
J. Am. Math. Soc., 14 (3) (2001), pp. 535--554.

%\bibitem[Bho]{bhosle:1992}
%{U. Bhosle}.
%{\it Generalized parabolic bundles and applications to torsion-free sheaves on nodal curves}. 
%Arkiv f\"ur Mat. (2) \textbf{30} (1992), 187--215.

%\bibitem[CKV]{CKV}
%S. Casalania-Martin, J. L. Kass and F. Viviani.
%{\it The local structure of compactified Jacobians}.
%Proc. London Math. Soc. (3) 110 (2015) 510--542.

%\bibitem[Cap]{caporaso}
%L. Caporaso. 
%{\it A compactification of the universal Picard variety over the moduli space of stable curves}. 
%J. Amer. Math. Soc. 7 (1994), no. 3, 589--660,


\bibitem{carbone}
R. M. Carbone.
{\it The direct image of generalized divisors and the Norm map between compactified Jacobians}.
Geom. Dedicata 216, 14 (2022).


%\bibitem[Coo1]{cook:1993}
%{P. R. Cook}. 
%{\it Local and Global Aspects of the Module theory of singular curves}.
%Ph.D. Thesis, University of Liverpool, 1993.

%\bibitem[Coo2]{cook:1998}
%{P. R. Cook}.
%{\it Compactified Jacobians and curves with simple singularities}. 
%Algebraic Geometry (Catania, 1993/Barcelona, 1994), 37--47, Lecture Notes in Pure and Appl. Math., \textbf{200}, Marcel Dekker, 1998.


\bibitem{donagi&pantev}
R. Donagi and T. Pantev.
{\it Langlands duality for Hitchin systems}.
Inventiones mathematicae 189, no. 3 (2012): 653--735.

\bibitem{drezet&narasimhan}
J.-M. Drezet and M. S. Narasimhan, 
{\it Groupe de Picard des variétés de modules de fibr\'es semi-stables sur les courbes alg\'ebriques}. 
Invent. Math., 97(1):53--94, 1989.

\bibitem{esteves} 
E. Esteves, 
{\it Compactifying the relative Jacobian over families of reduced curves.} 
Trans. Am. Math. Soc. \textbf{353} 8 (2001) 3045--3095.  

\bibitem{esteves&gagne&kleiman}
E. Esteves, M. Gagné and S. Kleiman.
{\it Autoduality of the compactified Jacobian}. 
J. Lond. Math. Soc. (2) 65 (2002), 3, 591--610.

\bibitem{esteves&kleiman}
E. Esteves and S. Kleiman. 
{\it The compactified Picard scheme of the compactified Jacobian}. 
Adv. Math. 198 (2005), 2, 484--503.

\bibitem{FGOP} 
E. Franco, P. Gothen, A. Oliveira and A. Pe\'on-Nieto.
{\it Unramified covers and branes on the Hitchin system.} 
Adv. Math. 377 (107493) (2021).

\bibitem{franco&jardim}
E. Franco, M. Jardim. {\it Mirror symmetry for Nahm branes.} Épijournal de Géométrie Algébrique 6 (2022).

\bibitem{franco&peon} E. Franco and A. Pe\'on-Nieto. {\it Branes on the singular locus of the Hitchin system from Borel and other parabolic subgroups}. Math. Nach. (2022). 

\bibitem{frenkel&witten}
{E. Frenkel, E. Witten}.
{\it Geometric endoscopy and mirror symmetry}. 
Commun. Number Theory Phys. \textbf{2} (2008), 113--283.

\bibitem{groechenig&shen}M. Groechenig and S. Shen. 
\textit{Complex K-theory of moduli spaces of Higgs bundles.} arXiv:2212.10695 (2022).


%\bibitem[Hai]{haiman}
%M. Haiman.
%{\it Hilbert schemes, polygraphs and the Macdonald positivity conjecture}.
%J. Amer. Math. Soc. 14 (2001), no. 4, 941--1006.

%\bibitem[Hau]{hausel}
%T. Hausel
%{ \it Global topology of the %Hitchin system.} 
%arXiv preprint arXiv:1102.1717 (2011).


\bibitem{hausel&hitchin} T. Hausel and N. Hitchin,  {\it Very stable Higgs bundles, equivariant multiplicities and mirror symmetry}, Inv. Math. (2022). 

\bibitem{hausel&pauly} 
T. Hausel and C. Pauly, 
{\it Prym varieties of spectral covers}. 
Geometry \& Topology \textbf{16} (2012), 1609--1638.

\bibitem{hausel&thaddeus}
T. Hausel, M.Thaddeus, 
{\it Mirror symmetry, Langlands duality, and the Hitchin system}. 
Inventiones Mathematicae (2003), 197--229.

\bibitem{HLS}
D. Hern\'andez Ruip\'erez, A.C. L\'opez Mart\'{\i}n and F. Sancho de Salas.
{\it Fourier–Mukai transforms for Gorenstein schemes}
Adv. Math., 211 (2) (2007), 594--620.

\bibitem{hitchin-self}
N. J. Hitchin.
{\it The self-duality equations on a Riemann surface}.
Proc. London Math. Soc.  {\bf 55} (1987), 59--126.

\bibitem{hitchin-charclasses} 
N. J. Hitchin, {\it Higgs bundles and characteristic classes},   Arbeitstagung Bonn 2013, {Progr. Math.}, \textbf{319}, Birkhäuser/Springer, Cham, 2016, 247--264. 

%\bibitem[Huy]{huybrechts}
D. Huybrechts.
{\it Fourier-Mukai Transforms in Algebraic Geometry}.
Oxford University Press (2006), ISBN: 978-0-19-929686-6.

\bibitem{kapustin&witten}
A. Kapustin, E. Witten, {\it Electric-magnetic duality and the geometric Langlands program}, {Comm. Number Theory Phys.} \textbf{1} (2007), 1--236.    

\bibitem{knudsen&mumford} 
F. Knudsen, D. Mumford.
{\it The projectivity of the moduli space of stable curves I: Preliminaries on “det” and “Div”}. 
Math. Scand. \textbf{39} (1976), 19--55.   

\bibitem{melo0} 
M. Melo, A. Rapagnetta and F. Viviani, 
{\it Fine compactified Jacobians of reduced curves.}
Trans. Amer. Math. Soc. 369 (2017), no. 8, 5341--5402.


\bibitem{melo1} 
M. Melo, A. Rapagnetta and F. Viviani. 
{\it  Fourier--Mukai and autoduality for compactified Jacobians. I}.
J. Reigne Angew. Math. 755 (2019) 1--65. 

\bibitem{melo2} 
M. Melo, A. Rapagnetta and F. Viviani.
{\it Fourier--Mukai and autoduality for compactified Jacobians. II.}
Geom. Topol. 23 (5) (2019), 2335--2395.

\bibitem{mukai}
S. Mukai.
{\it Duality between $D^{b}(X)$ and $D^b(\hat X)$ with its application to Picard sheaves. }
Nagoya Mathematical Journal, 81:153--175, 1981.

\bibitem{nevins}
T. Nevins.
{\it Descent of coherent sheaves and complexes to geometric invariant theory quotients}.
J. of Algebra (320) 6, Pages 2481--2495.


\bibitem{ngo}
B. C. Ng\^o, 
{\it Fibration de Hitchin et endoscopie}. 
{Invent. Math.} \textbf{164} (2006), 399--453.

\bibitem{simpson1}
C. T. Simpson,
{\it Moduli of representations of the fundamental group of a smooth projective variety I}.
Publ. Math. Inst. Hautes Etud. Sci. {\bf 79} (1994), 47--129.

%\bibitem[MS]{maulik&shen}
%D. Maulik and J. Shen. 
%{\it Endoscopic decompositions and the Hausel–Thaddeus conjecture}.
%Forum of Mathematics, Pi, 9, E8. (2021) doi:10.1017/fmp.2021.7

%\bibitem[Nee]{neeman}
%A. Neeman.
%{\it Triangulated Categories}.
%Princeton University Press, Annals of Mathematics Studies {\bf 148} (2001), ISBN: 978-0-691-08686-6.

\bibitem{Ngo}
B. C. Ng\^o.
{\it Le lemme fondamental pour les algèbres de Lie}, Publ. Math. Inst. Hautes Études Sci. 111 (2010), 1–169.

\bibitem{oda&seshadri}
T. Oda and C. S. Seshadri.
{\it Compactifications of the generalized Jacobian variety}
Trans. Amer. Math. Soc. 253 (1979), 1--90.

%\bibitem[Reg]{rego:1980}
%C. J. Rego.
%{\it The compactified Jacobian}.
%{Ann. Sci. \'Ec. Norm. Sup.} IV, \textbf{13}
%(1980), 211--223.

\bibitem{rizzardo}
A. Rizzardo.
{\it Adjoints to a Fourier--Mukai functor}.
Adv. Math., 322, 2017, 83--96.

%\bibitem[Saw]{sawon}
%J. Sawon. 
%{\it Twisted Fourier–Mukai transforms for holomorphic symplectic four-folds}.
%Adv. Math. 218 (2008), 3, 828--864.

\bibitem{schaub}
D. Schaub. 
\textit{spectrales et compactifications de jacobiennes.} 
Math. Z. 227 (1998), no. 2, 295–312.

\bibitem{simpson} 
C. Simpson.
{\it Moduli of representations of the fundamental group of a smooth projective variety. I}.
Publ. Math. Inst. Hautes Études Sci. 79 (1994), 47--129.


\end{thebibliography}
\end{document}